\documentclass[11pt]{article}

\usepackage{ifpdf}
\usepackage{graphicx, lineno}
\usepackage{verbatim}
\usepackage{graphicx}
\usepackage{amsmath,amssymb,amsfonts,euscript,enumerate,latexsym}
\usepackage{amsthm}
\usepackage{color,wrapfig}
\usepackage{pxfonts}
\usepackage{fancyhdr}
\usepackage{mathrsfs}
\usepackage{mathtools}
\usepackage{epsfig,subfigure}
\usepackage{marvosym}
\usepackage{txfonts}
\usepackage{hyperref}
\usepackage{times}
\usepackage{enumerate}
\usepackage{ulem}

\def\titlerunning#1{\gdef\titrun{#1}}
\makeatletter
\def\author#1{\gdef\autrun{\def\and{\unskip, }#1}\gdef\@author{#1}}
\def\address#1{{\def\and{\\\hspace*{18pt}}\renewcommand{\thefootnote}{}%
		\footnote {#1}}%
	\markboth{\autrun}{\titrun}}
\makeatother
\def\email#1{e-mail: #1}

\def\keywords#1{\par\medskip
	\noindent\textbf{Keywords.} #1}

\usepackage{etoolbox}
\makeatletter
\patchcmd{\@maketitle}
  {\ifx\@empty\@dedicatory}
  {\ifx\@empty\@date \else {\vskip3ex \centering\footnotesize\@date\par\vskip1ex}\fi
   \ifx\@empty\@dedicatory}
  {}{}
\patchcmd{\@adminfootnotes}
  {\ifx\@empty\@date\else \@footnotetext{\@setdate}\fi}
  {}{}{}
\makeatother

\newtheorem{thm}{Theorem}[section]
\newtheorem{cor}[thm]{Corollary}
\newtheorem{lemma}[thm]{Lemma}

\newcommand{\R}{{\mathbb{R}}}

\newcommand{\N}{{\mathbb{N}}}

\newcommand{\esssup}{{\mathrm{ess}\sup}}
\newcommand{\essinf}{{\mathrm{ess}\inf}}

\newcommand{\vp}{\varphi}
\newcommand{\osc}{\operatornamewithlimits{osc}}

\newcommand{\La}{\triangle}
\newcommand{\bs}{\backslash}

\begin{document}

\baselineskip=17pt
	
\titlerunning{}

\title{System of Porous Medium Equations }

\author{ Sunghoon Kim  
	\and Ki-Ahm Lee }

\date{}

\maketitle

\address{
	S. Kim (\Letter) :
	Department of Mathematics, School of Natural Sciences, The Catholic University of Korea,\\
	43 Jibong-ro, Wonmi-gu, Bucheon-si, Gyeonggi-do, 14662, Republic of Korea ;\\
	\email{math.s.kim@catholic.ac.kr}
	\and
	Ki-Ahm Lee :
	Department of Mathematical Sciences, Seoul National University, Gwanak-ro 1, Gwanak-Gu, Seoul, 08826, South Korea \&
	Korea Institute for Advanced Study, Seoul 02455, Korea;\\
	\email{ kiahm@snu.ac.kr }
}

\begin{abstract}
We investigate the evolution of population density vector, $\bold{u}=\left(u^1,\cdots,u^k\right)$, of $k$-species whose diffusion is controlled by its absolute value $\left|\bold{u}\right|$. More precisely we study the properties and asymptotic large time behaviour of solution $\bold{u}=\left(u^1,\cdots,u^k\right)$ of degenerate parabolic system
\begin{equation*}
\left(u^i\right)_t=\nabla\cdot\left(\left|\bold{u}\right|^{m-1}\nabla u^i\right) \qquad \mbox{for $m>1$ and $i=1,\cdots,k$}.
\end{equation*}
Under some regularity assumption, we prove that the function $u^i$ which describes the population density of $i$-th species with population $M_i$ converges to $\frac{M_i}{\left|\bold{M}\right|}\mathcal{B}_{\left|\bold{M}\right|}$ in space with two different approaches where $\mathcal{B}_{\left|\bold{M}\right|}$ is the Barenblatt solution of the porous medium equation with $L^1$-mass $\left|\bold{M}\right|=\sqrt{M_1^2+\cdots+M_k^2}$.\\
\indent As an application of the asymptotic behaviour, we establish a suitable Harnack type inequality which makes the spatial average of $u^i$ under control by the value of $u^i$ at one point.  We also find an 1-directional travelling wave type solutions and the properties of  solutions which has travelling wave behaviour at infinity.
\keywords{Degenerate Parabolic System, Asymptotic Behaviour, Harnack Type Inequality, Travelling Wave}
\end{abstract}

\setcounter{equation}{0}
\setcounter{section}{0}

\section{Introduction}\label{section-intro}
\setcounter{equation}{0}
\setcounter{thm}{0}

\indent In this paper we will investigate the properties and asymptotic behaviour of degenerate parabolic system which is formulated by population densities or distributions of different species in a closed system when they interact each other. More precisely, let $k\in\N$ be the number of different species and let $u^i\geq 0$, $\left(i=1,\cdots,k\right)$ represents the density of population of $i$-th species in that system. We consider the case that the evolution of $u^i$ is governed by some quantity depending on population densities of all species given as follow:
\begin{equation*}
\left|\bold{u}\right|=\sqrt{\sum_{i=1}^k\left(u^i\right)^2} \qquad \qquad \left(\,\bold{u}=\left(u^1,\cdots,u^k\right)\,\right)
\end{equation*}
\indent Now let us introduce the main system of this work which describes the evolution of the density function $u^i$:
\begin{equation}\label{eq-main-equation-of-system}\tag{SPME}
\begin{cases}
\begin{aligned}
\left(u^i\right)_t&=\nabla\left(m\left|\bold{u}\right|^{m-1}\nabla u^i\right) \qquad \mbox{in $\R^n\times\left(0,\infty\right)$}\\
u^i(x,0)&=u^i_0(x) \qquad \qquad \qquad \forall x\in\R^n
\end{aligned}
\end{cases}
\end{equation}
in the range of $m>1$, with initial data $u^i_0$ nonnegative, integrable and compactly supported.\\
\indent This  study of degenerate parabolic system has been motivated  by non-Newtonian fluids such as molten plastics, polymer and biological fluid, where the diffusion coefficient depends on the norm of the symmetric part of $\nabla u$. But they are expected to share similar properties including waiting time effects and methods for study as scalar Porous Medium Equations (shortly, PME) and parabolic p-Lpalace equation do. And a mixture of  isentropic gases through a porous
medium is expected to follow the system of PME.\\
\indent  There could be some possible candidates for the system of PME. At the first stage, we considered the evolution of population density of each species which is controlled by total population of all species in the same system, i.e., in \cite{KL1} we studied the local continuity and asymptotic behavior of solution $\bold{u}=\left(u^1,\cdots,u^k\right)$ of the following degenerated parabolic system
\begin{equation}\label{eq-main-eq-of-our-previous-work-govened-by-totall-populations}
\left(u^i\right)_t=\nabla\cdot\left(m\,U^{m-1}\nabla u^i\right) \qquad \mbox{for $m>1$ and $i=1,\cdot, k$}
\end{equation}
where the diffusion coefficients are given by total population density, $U=u^1+\cdots+u^k$. Recently, Kuusi, Monsaingeon and Videman \cite{KMV} showed the existence of a unique weak solution of \eqref{eq-main-eq-of-our-previous-work-govened-by-totall-populations} and derive regularity estimates under some assumptions.\\
\indent In this paper, we are going to investigate the properties of solution $\bold{u}=\left(u^1,\cdots,u^k\right)$ of \eqref{eq-main-equation-of-system}, whose diffusion coefficient depends on the norm of $u$  as that of non-Newtonian fluids depends on the norm of the symmetric part of $\nabla u$. \\
\indent As given above, the diffusion coefficients of the equation in \eqref{eq-main-equation-of-system} are composed of population densities of all species. Thus the solution of \eqref{eq-main-equation-of-system} could not diffuse in a way determined by solution itself, i.e., interaction between all species will affect each other and may make their population densities more and more similar as time goes. The interaction between species creates serious difficulties which are absent for the equation for a single species.\\  
\indent One of main difficulties in studying the properties and asymptotic behavior of \eqref{eq-main-equation-of-system} is the absence of comparison principle between corresponding components of two different solutions, not between components in one solution, i.e., for two solutions $\bold{u}=\left(u^1,\cdots,u^k\right)$ and $\bold{v}=\left(v^1,\cdots,v^k\right)$ of \eqref{eq-main-equation-of-system}, we could not guarantee
\begin{equation*}
u^i(x,t)\geq v^i(x,t) \qquad \mbox{for some $i\in\left\{1,\cdots,k\right\}$},
\end{equation*} 
even though all components of $\bold{u}$ are greater than corresponding components of $\bold{v}$ initially:
\begin{equation*}
u_0^i(x,t)\geq v_0^i(x,t) \qquad \forall 1\leq i\leq k.
\end{equation*} 
\indent In this paper, we are going to show the asymptotic large time behaviour of density function $u^i$, $\left(1\leq i\leq k\right)$, and we will also find Harnack type inequality for $u^i$, $\left(1\leq i\leq k\right)$, when all components of $\bold{u}=\left(u^1,\cdots,u^k\right)$ are equivalent to each other in the sense that
\begin{equation*}
c^{ij}u^j(x,t)\leq u^i(x,t) \leq C^{ij}u^j(x,t) \qquad \forall (x,t)\in\R^n\times[0,\infty),\,\,1\leq i,\,j\leq k
\end{equation*}
for some positive constants $C^{ij}>c^{ij}>0$.\\
\indent When $k=1$, the equation in \eqref{eq-main-equation-of-system} is called the porous medium equations or slow diffusion equation for $m>1$, heat equation for $m=1$ and fast diffusion equation for $m<1$. Large number of literatures on asymptotic behaviours and regularity theories of solutions of the fast diffusion and porous medium equations can be found. We refer the readers to the paper \cite{BBDGV, HK1, HK2, HKs, Va2} for the asymptotic behaviour, to the paper \cite{CF1, CF2, CF3, CVW, CW, DH, Di, HU, KL2} for regularity theory of the fast and slow diffusion equations.\\
\indent There are also many studies on the degenerate parabolic system with an equation similar to the one in \eqref{eq-main-equation-of-system}. In \cite{Yi}, they derived the degenerate parabolic system
\begin{equation}\label{eq-another-degenerate-parabolic-system-Laplace-type}
\bold{u}_t=\La\left(\left|\bold{u}\right|^{m-1}\bold{u}\right), \qquad m>1
\end{equation}
from a time-dependent $p$-curl system which describes Bean's critical-state model in the superconductivity theory and studied the existence and uniqueness of the global continuous solution. In \cite{ST}, authors investigated the boundedness and finite speed of propagation of solution of \eqref{eq-another-degenerate-parabolic-system-Laplace-type} i.e., they gave sharp estimates of $\left\|\,\left|\bold{u}\right|\,\right\|_{L^{\infty}}$ and the size of support of $\bold{u}$ in \eqref{eq-another-degenerate-parabolic-system-Laplace-type}.
\\
\indent The first part of this paper is started by assuming that all species are located separately from each other at first, i.e., suppose that
\begin{equation*}
\textbf{supp}\,u^i_0\,\cap \textbf{supp}\,u^j_0\,=\emptyset \qquad \forall 1\leq i<j\leq k.
\end{equation*}
Then the evolution of each species is not interfered by others on its support. Hence
\begin{equation*}
\left(u^i\right)_t=\nabla\left(m\,\left|\bold{u}\right|^{m-1}\nabla u^i\right)\qquad \Rightarrow \qquad \left(u^i\right)_t=\La \left(u^i\right)^m\qquad \mbox{on $\textbf{supp}\,u^i_0$}
\end{equation*}
for a short time, i.e., the function $u^i$ behaves like a solution of the standard porous medium equation in a short time. One of the interesting phenomenons caused by the degeneracy of porous medium equation is the finite speed of propagation. Thus there exists a waiting time $T>0$ until collision between species. The first observation of the first part is stated as follows.
\begin{thm}\label{thm-existence-of-waiting-time-if-initially-separated}
Let $\bold{u}=\left(u^1,\cdots,u^k\right)$ be a weak solution of \eqref{eq-main-equation-of-system} with initial data $u^i_0\in L^1\left(\R^n\right)\cap L^{m+1}\left(\R^n\right)$, $(1\leq i\leq k)$, nonnegative and compactly supported. Suppose that there exists a constant $d>0$ such that 
\begin{equation}\label{eq-infimum-distance-between-supports-of-u-i-s}
\textbf{dist}\left(\textbf{supp}\, u^i_0,\, \textbf{supp} \,u^j_0\right)\geq d>0 \qquad \forall 1\leq i,<j\,\leq k.
\end{equation}
Then there exists a constant $0<T\leq\infty$ such that
\begin{equation}\label{eq-waiting-time-of-mixing-of-isolated-supports}
\textbf{supp}\, u^i\left(\cdot,t\right)\cap \textbf{supp} \,u^j\left(\cdot,t\right)=\emptyset \qquad \forall 0\leq t\leq T,\,\,1\leq i<j\,\leq k.
\end{equation}
\end{thm}

By the first observation, touching between supports of $u^i$, $\left(1\leq i\leq k\right)$, will occur in a finite time if some conditions are imposed on the initial values such as non-degeneracy at the boundary. Question is how does the interaction between all species affect the diffusion of each species. More precisely, we'd like to know how fast two different species diffuse in each other's territory when they begin to mix. The second observation of the first part is stated as follows.

\begin{thm}\label{eq-lem-equailty-of-support-ofu-i-and-u-j-when-initially-joined}
Let $m>1$ and $t_0\geq 0$. Let $\bold{u}=\left(u^1,\cdots,u^k\right)$ be a continuous solution of
\begin{equation}\label{eq-system-of-PME-on-t_0-to-infty}
\left(u^i\right)_t=\nabla\cdot\left(m\left|\bold{u}\right|^{m-1}\nabla u^i\right) \qquad \forall (x,t)\in\R^n\times(t_0,\infty),\,\,1\leq i\leq k
\end{equation} 
with $u^i\geq 0$ for all $1\leq i\leq k$. Suppose that 
\begin{equation*}
\left|\textbf{supp}\,u^j(\cdot,t_0)\cap \textbf{supp}\,u^l(\cdot,t_0)\right|\neq 0
\end{equation*}
for some $j$, $l\in\left\{1,\cdots,k\right\}$. Then
\begin{equation}\label{eq-statements-of-equality-of-supprot-of-u-i-s-forall-i}
\textbf{supp}\,u^j(t)=\textbf{supp}\,u^{l}(t)\qquad \forall t>t_0.
\end{equation}
\end{thm}
This extension plays an important role to explain the interaction between different species after touching of supports. \\
\indent The main result of the first part is the asymptotic large time behaviour of function $u^i$, $\left(1\leq i\leq k\right)$. In most cases on Cauchy problem, the solution of the problem loses the informations given by its initial data after a long time and evolves only by the equation of that problem. Thus, it can be expected that all of functions $u^i$, $\left(1,\cdots,k\right)$ in \eqref{eq-main-equation-of-system} behave similarly after a large time since they all are controlled by the same diffusion coefficients. More precisely, if functions $u^1$, $\cdots$, $u^k$ maintain their $L^1$-mass, we can expect that there exists a function $f$ such that
\begin{equation}\label{eq-expected-similarity-of-u-i-by-c-i-and-some-function-f}
u^i(x,t) \approx c^if(x,t) \qquad \mbox{after a large time}
\end{equation}
for some constant $c^1$, $\cdots$, $c^k$ determined by $L^1$ norm of $u^1$, $\cdots$, $u^k$, respectively. Moreover, by \eqref{eq-main-equation-of-system} and \eqref{eq-expected-similarity-of-u-i-by-c-i-and-some-function-f}, the function $f$ will satisfy
\begin{equation}\label{eq-PME-satisfied-by-f-that-is-similar-with-u-i-after-a=long=time}
f_t=\left(\sum_{i=1}^k\left(c^i\right)^2\right)^{\frac{m-1}{2}}\La f^m \qquad \mbox{after a large time}.
\end{equation}
For a constant $M>0$, let $\mathcal{B}_{M}$ be the self-similar Barenblatt solution of the porous medium equation with $L^1$-mass $M$. Then it is given explicitly by
\begin{equation}\label{eq-barenblatt-solution-of-PME}
\mathcal{B}_M(x,t)=t^{-a_1}\left(\mathcal{C}_M-\frac{a_3|x|^2}{t^{2a_2}}\right)_+^{\frac{1}{m-1}}
\end{equation} 
where the constant $\mathcal{C}_M>0$ is related to the $L^1$-mass $M$ of barenblatt solution and
\begin{equation}\label{eq-constant-alpha-1-beta-1-k}
a_1=\frac{n}{(m-1)n+2}, \qquad a_2=\frac{a_1}{n}\qquad \mbox{and} \qquad a_3=\frac{a_1(m-1)}{2mn}. 
\end{equation}
It is well known by \cite{LV, Va1} that the solution $f$ of \eqref{eq-PME-satisfied-by-f-that-is-similar-with-u-i-after-a=long=time} converges to some Barenblatt  solution $\mathcal{B}_{M}$ of the porous medium equation as $t\to\infty$. Thus we have no choice but to think that 
\begin{equation*}
u^i\approx c^if\to c^i\mathcal{B}_{M} \qquad \mbox{as $t\to\infty$}.
\end{equation*}
Under this expectation, the main result of the first part is as follows.
\begin{thm}\label{thm-asymptotic-behaviour-of-system-PME-L-1=and=L=infty}
Let $m>1$ and let $\bold{u}=\left(u^1,\cdots,u^k\right)$ be a non-negative solution of \eqref{eq-main-equation-of-system} with initial data $u^i_0$, $(1\leq i\leq k)$, integrable and compactly supported. Then
\begin{equation}\label{eq-L-1-convergence-between-u-i-and-M-i-barrenblatt-profiles}
\lim_{t\to\infty}\left\|u^i(\cdot,t)-\frac{M_i}{\left|\bold{M}\right|}\mathcal{B}_{\left|\bold{M}\right|}\left(\cdot,t\right)\right\|_{L^1}=0
\end{equation}
and
\begin{equation}\label{eq-L-infty-convergence-between-u-i-and-M-i-barrenblatt-profiles}
\lim_{t\to\infty}t^{a_1}\left|u^i(\cdot,t)-\frac{M_i}{\left|\bold{M}\right|}\mathcal{B}_{\left|\bold{M}\right|}\left(\cdot,t\right)\right|=0\qquad \mbox{uniformly on every compact subset of $\R^n$}
\end{equation}
for some constant $a_1>0$ where $\left\|u_0^i\right\|_{L^1\left(\R^n\right)}=M_i$ for $i=1,\cdots, k$ and $\bold{M}=\left(M_1,\cdots,M_k\right)$.
\end{thm}

In the end of the first part, we also studied the asymptotic large time behaviour of function $u^i$, $\left(1\leq i\leq k\right)$, through another approach called the entropy method.

\begin{thm}\label{thm-asymptotic-with-entrophy-methods}
Let $m>1$ and let $\bold{u}=\left(u^1,\cdots,u^k\right)$ be a solution of \eqref{eq-main-equation-of-system} with initial data $u^i_0$, $(1\leq i\leq k)$, integrable and compactly supported. For any $M>0$, let
\begin{equation*}
\widetilde{\mathcal{B}_{M}}\left(\eta\right)=t^{\,a_1}\mathcal{B}_{M}\left(t^{a_2}\eta,t\right)=\left(\mathcal{C}_{M}-\frac{a_1(m-1)}{2mn}\left|\eta\right|^2\right)^{\frac{1}{m-1}}_+.
\end{equation*} 
Then rescaled functions 
\begin{equation*}
\theta^{\,i}\left(\eta,\tau\right)=t^{\,a_1}u^i\left(x,t\right) \qquad\left(\,\eta=xt^{-a_2},\,\,\,\tau=\log t\,\right)
\end{equation*} 
converges uniformly on any compact subset of $\R^n$ and also in $L^1\left(\R^n\right)$ to the function $\frac{M_i}{\left|\bold{M}\right|}\widetilde{\mathcal{B}_{\left|\bold{M}\right|}}$ as $\tau\to\infty$  where $\left\|u_0^i\right\|_{L^1\left(\R^n\right)}=M_i$ for $i=1,\cdots, k$ and $\bold{M}=\left(M_1,\cdots,M_k\right)$.
\end{thm}

In the second part of the paper, as an application of asymptotic large time behaviour (\,Theorem \ref{thm-asymptotic-behaviour-of-system-PME-L-1=and=L=infty} and Theorem \ref{thm-asymptotic-with-entrophy-methods}\,), we establish a suitable Harnack type inequality for the function $u^i$ of \eqref{eq-main-equation-of-system} which makes the spatial average of $u^i$ under control by the value of $u^i$ at one point. The result is stated as follows.
\begin{thm}\label{lem-calculation-about-waiting-time-lower-bound}
Let $m>1$ and let $\bold{u}=\left(u^1,\cdots,u^k\right)$ be a continuous weak solution of 
\begin{equation}\label{eq-system-PME-for-waiting-time-1}
\left(u^i\right)_t=\nabla\left(m\left|\bold{u}\right|^{m-1}\nabla u^i\right) \qquad \mbox{in $\R^n\times\left[0,T\right]$}.
\end{equation}
with the initial data $u^i_0$, $(1\leq i\leq k)$, non-negative, integrable and compactly supported. Suppose that 
\begin{equation*}
u^i(x,t)\geq 0 \qquad \forall (x,t)\in\R^n\times\left[0,T\right],\,\,1\leq i\leq k
\end{equation*}
and there exists uniform constant $\mu_0>0$ such that
\begin{equation*}
\frac{\min_{1\leq l\leq k}\left\{\int_{\R^n}u^l(x,0)\,dx\right\}}{\max_{1\leq l\leq k}\left\{\int_{\R^n}u^l(x,0)\,dx\right\}}\geq\mu^0>0.
\end{equation*}
Then for $R>T^{\frac{1}{2}}>0$ there exists a constant $C^{\ast}=C^{\ast}\left(m,n\right)>0$ such that
\begin{equation}\label{eq-claim-L-1-norm=of-u-at-initial-is-bounded-by-H-sub-m}
\int_{|x|<R}u^i(x,0)\,dx\leq \frac{C^{\ast}}{\left(\mu^i\right)^{\frac{(m-1)n+2}{2}}}\left(\frac{R^{n+\frac{2}{m-1}}}{T^{\frac{1}{m-1}}}+T^{\frac{n}{2}}\left(u^i\right)^{\frac{(m-1)n+2}{2}}\left(0,T\right)\right) \qquad \forall 1\leq i\leq k
\end{equation}
where
\begin{equation*}
\mu^i=\frac{\int_{\R^n}u^i(x,0)\,dx}{\max_{1\leq l\leq k}\left\{\int_{\R^n}u^l(x,0)\,dx\right\}} \qquad \forall 1\leq i\leq k.
\end{equation*}
\end{thm}

Lastly, in the third part of this paper, we investigate the properties of 1-directional travelling wave type solution which is one of the simplest solutions whose all components have unbounded $L^1$-mass.  More precisely, let $0\leq l\leq k$ be a integer and let $\bold{u}=\left(u^1,\cdots,u^k\right)$ be a 1-directional travelling wave type solution with respect to the direction $e\in\R$ of 
\begin{equation*}\label{eq-main-system-eq-for-travelling-wave-type-with-direction-e}
\left(u^i\right)_t=\nabla\cdot\left(m\left|\bold{u}\right|^{m-1}\nabla u^i\right) \qquad \forall (x,t)\in\R^n\times\R,
\end{equation*}
i.e., the functions $u^i$, $\left(1\leq i\leq k\right)$, can be given in the form of
\begin{equation*}\label{eq-expression-of-soluiton-u-to-i-by-g-single-variable-function-travelling-wave-eq-with-direction-e}
u^i(x,t)=g^i\left(x\cdot e+c^it\right) \qquad \forall x\in\R,\,\,1\leq i\leq k
\end{equation*}
for some constants $c^1$, $\cdots$, $c^k$ where
\begin{equation*}\label{eq-assumption-for-constant-in-1-dimensional-solution-with-direction-e}
c^1,\cdots, c^l<0\qquad \mbox{and} \qquad c^{l+1},\cdots, c^k>0
\end{equation*}
and single variable functions $g^1$, $\cdots$, $g^k$ where
\begin{equation*}\label{eq-condition-of-g-i-1-first-l-th-functions-with-direction-e}
g^{i_1}(s)=0 \qquad  \mbox{for $s\geq 0$}, \qquad 
g^{i_1}(s)>0\qquad  \mbox{for $s<0$} \qquad \forall 1\leq i_1\leq l 
\end{equation*}
and
\begin{equation*}\label{eq-condition-of-g-i-1-last-k-l-th-functions-with-direction-e}
g^{i_2}(s)>0\qquad \mbox{for $s>0$}\qquad g^{i_2}(s)=0 \qquad \mbox{for $s\leq 0$} \qquad \forall l+1\leq i_2\leq k .
\end{equation*}
Then we first showed that there are only 1-directional travelling wave type solutions with components whose directions are the same. In addition, we find the explicit form of the 1-dimensional travelling wave type solution. The statement is as follows.
\begin{thm}\label{cor-expression-of-travelling-wave-equation-explicit-form}
Let $\bold{u}=\left(u^1,\cdots,u^k\right)$ be a 1-directional travelling wave solution of \eqref{eq-main-system-eq-for-travelling-wave-type} with respect to the direction $e\in\R$. Then
\begin{equation}\label{eq-form-of-1-directional-travelling-wave-type-sol-1}
u^i(x,t)=c_i\left(\widehat{c}\,t-x\cdot e\right)_+^{\frac{1}{m-1}} \qquad \forall (x,t)\in\R^n\times\R,\,\,1\leq i\leq k
\end{equation}
for some positive constants $c_1$, $\cdots$, $c_k$ and $\widehat{c}$ where
\begin{equation*}
\widehat{c}=\frac{1}{m-1}\left(\sum_{i=1}^kc_i^{\,2}\right)^{\frac{m-1}{2}}
\end{equation*}
or
\begin{equation}\label{eq-form-of-1-directional-travelling-wave-type-sol-2}
u^i(x,t)=\overline{c}_i\left(\widetilde{c}\,t+x\cdot e\right)_+^{\frac{1}{m-1}} \qquad \forall (x,t)\in\R^n\times\R,\,\,1\leq i\leq k
\end{equation}
for some positive constants $\overline{c}_1$, $\cdots$, $\overline{c}_k$ and $\widetilde{c}$ where
\begin{equation*}
\widetilde{c}=\frac{1}{m-1}\left(\sum_{i=1}^k\overline{c}_i^{\,2}\right)^{\frac{m-1}{2}}.
\end{equation*}
\end{thm}
By scaling and an argument similar to the proof of Theorem \ref{cor-expression-of-travelling-wave-equation-explicit-form}, we can also find the properties of the 1-directional travelling wave-like solutions whose behaviour at infinity is like the travelling wave.\\
\indent We end up this section by stating the definition of solutions. We say that $\bold{u}=\left(u^1,\cdots,u^k\right)$ is a weak solution of \eqref{eq-main-equation-of-system} in $\R^n\times(0,\infty)$ if $u^i$, $\left(i=1,\cdots,k\right)$, is a locally integrable function satisfying, for any $T>0$
\begin{enumerate}
\item Function space of $u^i$, $\left(1,\cdots,k\right)$:
\begin{equation*}\label{eq-first-condition-of-weak-soluiton-u-with-U}
\left|\bold{u}\right|^{m-1}\left|\nabla u^i\right|\in L^2\left(0,T:L^2\left(\R^n\right)\right) 
\end{equation*}
\item Weak formula of $u^i$, $\left(1,\cdots,k\right)$: 
\begin{equation}\label{eq-identity--of-formula-for-weak-solution}
\int_{0}^{T}\int_{\R^n}\left\{m\left|\bold{u}\right|^{m-1}\nabla u^i\cdot\nabla\vp-u^i\vp_t\right\}\,dxdt=\int_{\R^n}u^i_0(x)\vp(x,0)\,dx
\end{equation}
holds for any test function $\vp\in H^1\left(0,T:L^2\left(\R^n\right)\right)\cap L^2\left(0,T:H^1_0\left(\R^2\right)\right)$ which vanishes for $t=T$.
\end{enumerate}

\indent A brief outline of the paper is as follows. In Section 2, we will give a series of preliminary such as the existence and local continuity of functions $u^i$, $\left(1\leq i\leq k\right)$. Section 3 is devoted to the proofs of Theorem \ref{thm-existence-of-waiting-time-if-initially-separated}, Theorem \ref{eq-lem-equailty-of-support-ofu-i-and-u-j-when-initially-joined}, Theorem \ref{thm-asymptotic-behaviour-of-system-PME-L-1=and=L=infty} and Theorem \ref{thm-asymptotic-with-entrophy-methods}. Section 4 deals with the regularity properties of continuous weak solution of \eqref{eq-main-equation-of-system}. We establish that continuous weak solution has the Harnack type inequality if there some relation between components of the solution $\bold{u}$, (\,Theorem \ref{lem-calculation-about-waiting-time-lower-bound}\,). The last Section 5 is devoted to the study of 1-directional travelling wave type solutions and solutions which has travelling wave behaviour at infinity.

\section{Preliminary Results}
\setcounter{equation}{0}
\setcounter{thm}{0}

We will study various preliminary results of solution $\bold{u}$ of \eqref{eq-main-equation-of-system} which will be used in the next sections of the paper.

\subsection{Existence of solution $\bold{u}=\left(u^1,\cdots,u^k\right)$}

As the first result of this paper, we are going to deal with the existence of solution of \eqref{eq-main-equation-of-system}. We start by collecting a few properties of solution $\bold{u}$.\\
\indent For a positive, integrable function $\mathcal{A}(x,t)$, we consider the following problem 
\begin{equation}\label{eq-equation-of-system-with-given-bold-u-absolute}
\begin{cases}
\begin{aligned}
u_t&=\nabla\left(\mathcal{A}\,\nabla u\right) \qquad \mbox{in $\R^n\times\left(0,\infty\right)$}\\
u(x,0)&=u_0(x) \qquad \qquad \forall x\in\R^n.
\end{aligned}
\end{cases}
\end{equation}
Then, we say that $u$ is a weak solution of \eqref{eq-equation-of-system-with-given-bold-u-absolute} in $\R^n\times(0,\infty)$ if $u$ is a locally integrable function satisfying, for any $T>0$
\begin{enumerate}
	\item Function space of $u$:
	\begin{equation*}
    \mathcal{A}\left|\nabla u\right|\in L^2\left(0,T:L^2\left(\R^n\right)\right) 
	\end{equation*}
	\item Weak formula of $u$: 
	\begin{equation*}
	\int_{0}^{T}\int_{\R^n}\left\{\mathcal{A}\nabla u\cdot\nabla\vp-u\vp_t\right\}\,dxdt=\int_{\R^n}u_0(x)\vp(x,0)\,dx
	\end{equation*}
	holds for any test function $\vp\in H^1\left(0,T:L^2\left(\R^n\right)\right)\cap L^2\left(0,T:H^1_0\left(\R^2\right)\right)$ which vanishes for $t=T$.
\end{enumerate}
By an argument similar to the proof of Lemma 2.3 of \cite{KL1}, we can get the following Uniqueness Lemma for the weak solution $u$. 
\begin{lemma}[Uniqueness of solutions]\label{lem-uniquness-of-weak-solution-2}
The Problem \eqref{eq-equation-of-system-with-given-bold-u-absolute} has at most one solution if $u_0\in L^2\left(\R^n\right)$.
\end{lemma}

Let $\bold{u}=\left(u^1,\cdots,u^k\right)$ be solutions of \eqref{eq-main-equation-of-system}  and suppose that $u^i$ has second differentiable in space variable $x$ and first differentiable in time variable $t$ for all $1\leq i\leq k$.	Then, by simple computation we have
\begin{equation}\label{eq-computation-for-equation-of-bold-u-square-1}
\left[\left(u^{i}\right)^2\right]_t=2u^iu^i_t=2u^i\nabla\left(m\left|\bold{u}\right|^{m-1}\nabla u^i\right) \qquad \forall 1\leq i\leq k
\end{equation}
and
\begin{equation}\label{eq-computation-for-equation-of-bold-u-square-2}
\nabla\left(m\left|\bold{u}\right|^{m-1}\nabla \left(u^{i}\right)^2\right)=2u^i\nabla\left(m\left|\bold{u}\right|^{m-1}\nabla u^i\right)+2m\left|\bold{u}\right|^{m-1}\left|\nabla u^i\right|^2 \qquad \forall  1\leq i\leq k.
\end{equation}
Denote by 
\begin{equation*}
w=\left|\bold{u}\right|^2=\sum_{i=1}^{k}\left(u^i\right)^2.
\end{equation*} 
Then by \eqref{eq-computation-for-equation-of-bold-u-square-1} and \eqref{eq-computation-for-equation-of-bold-u-square-2},
\begin{equation}\label{eq-equation-for-bold-w-equal-standard-norm}
w_t=\nabla\left(mw^{\frac{m-1}{2}}\nabla w\right)-F\qquad \forall (x,t)\in\R^n\times\left(0,\infty\right)
\end{equation}
where
\begin{equation*}
F(x,t)=2m\left|\bold{u}(x,t)\right|^{m-1}\sum_{i=1}^{k}\left|\nabla u^i(x,t)\right|^2 \qquad \forall (x,t)\in\R^n\times(0,\infty).
\end{equation*}
Let 
\begin{equation*}
\overline{w}(x,t)=w^{\frac{1}{2}}(x,t)=\left|\bold{u}\right|(x,t) \qquad \forall (x,t)\in\R^n\times[0,\infty).
\end{equation*}
By Cauchy-Schwarz inequality,
\begin{equation}\label{eq-inequalities-from-Cauchy-Schwarz-inequality-1}
\left|\nabla \overline{w}\right|^2=\frac{\left|\nabla w\right|^2}{4w}\leq \sum_{i=1}^k\left|\nabla u^i\right|^2.
\end{equation}
Thus, by \eqref{eq-equation-for-bold-w-equal-standard-norm} and \eqref{eq-inequalities-from-Cauchy-Schwarz-inequality-1} we can expect that $\overline{w}$ satisfies
\begin{align}
\overline{w}_{\,t}=\nabla\left(m\overline{w}^{\,m-1}\nabla\,\overline{w}\right)+m\overline{w}^{\,m-2}\left(\left|\nabla\,\overline{w}\right|^2-\sum_{i=1}^k\left|\nabla u^i\right|^2\right)\leq \La\overline{w}^{\,m} \qquad \forall (x,t)\in\R^n\times\left(0,\infty\right)\label{eq-equation-for-bold-v-sub-equal-standard-PME}.
\end{align}
To check whether  \eqref{eq-equation-for-bold-v-sub-equal-standard-PME} is also true for the weak solution, we take the function $\frac{u^i}{\overline{w}}\vp$ as a test function in \eqref{eq-identity--of-formula-for-weak-solution} and sum it over $i=1,\cdots,k$. Then the following weak formulation can be easily derived.
\begin{equation}\label{eq-equation-for-bold-v-sub-equal-in-the-weak-sense-standard-PME}
\int_{0}^{T}\int_{\R^n}\left\{m\overline{w}^{m-1}\nabla \overline{w}\cdot\nabla\vp-m\overline{w}^{\,m-2}\left(\left|\nabla\,\overline{w}\right|^2-\sum_{i=1}^k\left|\nabla u^i\right|^2\right)\vp-\overline{w}\vp_t\right\}\,dxdt=\int_{\R^n}\overline{w}(x,0)\vp(x,0)\,dx.
\end{equation}
By \eqref{eq-equation-for-bold-v-sub-equal-in-the-weak-sense-standard-PME}, $\overline{w}=\left|\bold{u}\right|$ is a subsolution of the standard porous medium equation, i.e.,
\begin{equation*}
\overline{w}_t\leq \La\overline{w}^m \qquad \mbox{in the weak sense}.
\end{equation*}
 Thus, by an argument similar to the proofs in Chapter 9 of \cite{Va1} we have the following lemma for solution $u^i$, $\left(1\leq i\leq k\right)$.

\begin{lemma}[\textbf{$L^{\infty}$ Bound of $\left|\bold{u}\right|$}]\label{eq-lemma-L-infty-bound-of-bold-u-absolute-value}
Let $m>1$. Let $\bold{u}=\left(u^1,\cdots,u^k\right)$ be a weak solution of \eqref{eq-main-equation-of-system} with initial data $u_0^i\in L^1\left(\R^n\right)$ compactly supported in a bounded set of $\R^n$. Let 
\begin{equation*}
v(x,t)=\left|\bold{u}\right|(x,t) \qquad \mbox{and} \qquad v_0(x)=\sqrt{\sum_{i=1}^k\left(u^{i}_0(x)\right)^2}. 
\end{equation*}
Then
\begin{equation*}
v(x,t)\leq C\left\|v_0\right\|^{a_2}_{L^1}t^{-a_1}
\end{equation*}
where $a_1=\frac{n}{n(m-1)+2}$, $a_2=\frac{2}{n(m-1)+2}$ and $C>0$ depending only on $m$ and $n$.
\end{lemma}

One of the interesting phenomenons caused by the degeneracy of porous medium equation is the finite speed of propagation. Thus we can have the following lemma.
\begin{lemma}\label{eq-lem-boundedness-of-support-of-solution}
Let $m>1$. Let $\bold{u}=\left(u^1,\cdots,u^k\right)$ be a weak solution of \eqref{eq-main-equation-of-system} with initial data $u_0^i\in L^1\left(\R^n\right)$ compactly supported in a bounded set of $\R^n$. Then for every $t>0$, there exists a constant $R=R(t)>0$ such that
\begin{equation*}
\textbf{supp}\, u^i\left(\cdot,t\right)\subset \textbf{supp}\, \left|\bold{u}\right|\left(\cdot,t\right)\subset B_{R} \qquad \forall 1\leq i\leq k,\,\,t>0.
\end{equation*}
Moreover, $\left|\bold{u}\right|\left(\cdot,t\right)$ has the finite speed of propagation for all $t\geq 0$.
\end{lemma}

Let $u_0^i\in L^1\left(\R^n\right)\cap L^{m+1}\left(\R^n\right)$ be compactly supported in a bounded set of $\R^n$.  Then by Lemma \ref{eq-lem-boundedness-of-support-of-solution}, the function $u^i$ is also supported in a bounded set of $\R^n$ for all $t>0$. Multiplying \eqref{eq-equation-for-bold-w-equal-standard-norm} by $w^{\frac{m-1}{2}}$ and integrating over $\R^n\times(0,t)$, we have
\begin{equation}\label{eq-energy-type-inequality-of-bold-w-equal-standard-norm-of-bold-u}
\begin{aligned}
\frac{2}{m+1}\int_{\R^n}w^{\frac{m+1}{2}}(x,t)\,dx&+\frac{2(m-1)}{m}\int_{0}^{t}\int_{\R^n}\left|\nabla w^{\frac{m}{2}}\right|^2\,dxd\tau\\
&+\int_{0}^{t}\int_{\R^n}w^{\frac{m-1}{2}}F\,dxd\tau\leq \frac{2}{m+1}\int_{\R^n}w^{\frac{m+1}{2}}(x,0)\,dx, \qquad \forall t>0.
\end{aligned}
\end{equation}

By \eqref{eq-energy-type-inequality-of-bold-w-equal-standard-norm-of-bold-u}, we now are ready to obtain the function space where the functions $u^i$, $\left(1\leq i\leq k\right)$, belong.

\begin{lemma}[\textbf{Function Space}]\label{lem-space-where-weak-solution-belongs-to-1}
Let $m>1$. Let $\bold{u}=\left(u^1,\cdots,u^k\right)$ be a weak solution of \eqref{eq-main-equation-of-system} with initial data $u_0^i\in L^1\left(\R^n\right)\cap L^{m+1}\left(\R^n\right)$ compactly supported in a bounded set of $\R^n$. Then
\begin{equation*}
\left|\bold{u}\right|^{m-1}\left|\nabla u^i\right|\in L^2\left(0,\infty:L^2\left(\R^n\right)\right) \qquad \forall 1\leq i\leq k.
\end{equation*}
\end{lemma}
\begin{proof}
Let $w=\left|\bold{u}\right|^2$. Since $u_0^i\in L^{m+1}\left(\R^n\right)$ for all $1\leq i\leq k$, we have
\begin{equation}\label{eq-initial-estimates-for-bold-w-1}
w(x,0)\in L^{\frac{m+1}{2}}\left(\R^n\right).
\end{equation}
By \eqref{eq-energy-type-inequality-of-bold-w-equal-standard-norm-of-bold-u} and \eqref{eq-initial-estimates-for-bold-w-1},
\begin{equation*}
\int_{0}^{\infty}\int_{\R^n}\left(w^{\frac{m-1}{2}}\left|\nabla u^i\right|\right)^2\,dxdt \leq\sum_{i=1}^{k}\int_{0}^{\infty}\int_{\R^n}\left(w^{\frac{m-1}{2}}\left|\nabla u^i\right|\right)^2\,dxdt \leq \frac{2}{m\left(m+1\right)}\left\|w(\cdot,0)\right\|^{\frac{2}{m+1}}_{L^{\frac{m+1}{2}}\left(\R^n\right)}<\infty \qquad \forall 1\leq i\leq k
\end{equation*}
and the lemma follows.
\end{proof}

As a consequence of Lemma \ref{lem-space-where-weak-solution-belongs-to-1},  we can have the $L^1$ mass conservation. The proof is almost same as one of Lemma 2.6 of \cite{KL1}.
\begin{lemma}[\textbf{Mass Conservation}]\label{lem-conservation-law-of-L-1-mass}
Let $m>1$ and let $\bold{u}=\left(u^1,\cdots,u^k\right)$ be a weak solution of \eqref{eq-main-equation-of-system}  with initial data $u_0^i\in L^1\left(\R^n\right)\cap L^{m+1}\left(\R^n\right)$. Then for every $t>0$, we have
\begin{equation*}
\int_{\R^n}u^i(x,t)\,dx=\int_{\R^n}u_0^i\,dx \qquad \forall t>0,\,\,1\leq i\leq k.
\end{equation*}
\end{lemma}

We now are ready for the existence of weak solution of \eqref{eq-main-equation-of-system}.
\begin{thm}[\textbf{Existence}]\label{thm-existence-of-weak-solution-u-i-of-system}
Let $m>1$. For each $1\leq i\leq k$, let $u_0^i\in L^1\left(\R^n\right)\cap L^{m+1}\left(\R^n\right)$ be a function compactly supported in a bounded set of $\R^n$. Then there exists a weak solution $\bold{u}=\left(u^1,\cdots,u^k\right)$ of \eqref{eq-main-equation-of-system}.
\end{thm}

\begin{proof}
For any $M>1$ and $0<\epsilon<1$, let
\begin{equation*}
u^i_{0,M}(x)=\min\left(u^i_0(x),M\right) \qquad \forall x\in\R^n,\,\,1\leq i\leq k.
\end{equation*} 
Then by standard theory for non-degenerate parabolic equation \cite{LSU}, for any $M>1$ and $0<\epsilon<1$ there exists the solution $u^i_{\epsilon,M}$, $\left(1\leq i\leq k\right)$, of 
\begin{equation}\label{eq-for-existence-non-degenerate-pde-with-bold-w-epsilon-M-diffusion-coefficients}
\begin{cases}
\begin{aligned}
u_t&=\nabla\cdot\left(m\left(w_{\epsilon,M}^{\frac{m-1}{2}}+\epsilon\right)\nabla u\right) \qquad \mbox{in $\R^n\times(0,\infty)$}\\
u(x,0)&=u^i_{0,M}(x) \qquad \qquad  \forall x\in\R^n
\end{aligned}
\end{cases}
\end{equation}
where
\begin{equation*}
w_{\epsilon,M}(x,t)=\sum_{i=1}^{k}\left(u^{i}_{\epsilon,M}(x,t)\right)^2 \qquad \forall (x,t)\in\R^n\times\left[0,\infty\right).
\end{equation*}
By direct computation, $w_{\epsilon,M}$ satisfies
\begin{align}
\left(w_{\epsilon,M}\right)_t&=\nabla\cdot\left(m\left(w^{\frac{m-1}{2}}_{\epsilon,M}+\epsilon\right)\nabla w_{\epsilon,M}\right)-2m\left(w_{\epsilon,M}^{\frac{m-1}{2}}+\epsilon\right)\sum_{i=1}^{k}\left|\nabla u^i_{\epsilon,M}(x,t)\right|^2\notag\\
&\leq\nabla\cdot\left(mw^{\frac{m-1}{2}}_{\epsilon,M}\nabla w_{\epsilon,M}\right)+\epsilon m\La w_{\epsilon,M}-F \qquad \mbox{in $\R^n\times(0,\infty)$} \label{eq-for-bold-w-epsilon-M}
\end{align}
where
\begin{equation*}
F(x,t)=2mw_{\epsilon,M}^{\frac{m-1}{2}}\sum_{i=1}^{k}\left|\nabla u^i_{\epsilon,M}(x,t)\right|^2\geq 0 \qquad \forall (x,t)\in\R^n\times(0,\infty).
\end{equation*}
By the non-degeneracy of diffusion coefficients $w_{\epsilon,M}$, one can easily check that 
\begin{equation*}
0\leq u^i_{\epsilon,M}(x,t)\leq M \qquad \forall (x,t)\in\R^n\times[0,\infty)
\end{equation*}
and
\begin{equation*}
u^i_{\epsilon,M}(x,t)\to 0 \qquad \forall t>0\qquad \mbox{as $|x|\to\infty$}.
\end{equation*}
Multiplying the first equation of \eqref{eq-for-existence-non-degenerate-pde-with-bold-w-epsilon-M-diffusion-coefficients} by $u^i_{\epsilon,M}$ and integrating over $\R^n\times(0,\infty)$, we have
\begin{equation}\label{eq-estimates-of-L-2-of-u-i-and-H-1-of--sqrt-epsilon-u-i}
\begin{aligned}
\frac{1}{2}\int_{\R^n}\left(u_{\epsilon,M}^i(x,t)\right)^2\,dx+m\int_{0}^{\infty}\int_{\R^n}\left(\epsilon^{\frac{1}{2}}\left|\nabla u^i_{\epsilon,M}\right|\right)^2\,dxdt\leq \frac{1}{2}\int_{\R^n}\left(u_0^i\right)^2\,dx \qquad \forall t>0.
\end{aligned}
\end{equation}
By an argument similar to the proof of Lemma \ref{lem-space-where-weak-solution-belongs-to-1}, there exists an uniform constant $C>0$ such that
\begin{equation}\label{eq-estimates-for-u-epsilon-M-similar-to-weak-solution-space}
\int_0^{\infty}\int_{\R^n}\left|\nabla w_{\epsilon,M}^{\frac{m}{2}}\right|^2\,dxdt+\int_0^{\infty}\int_{\R^n}\left(w_{\epsilon,M}^{\frac{m-1}{2}}\left|\nabla u_{\epsilon,M}^i\right|\right)^2\,dxdt\leq C\sum_{i=1}^k\left\|u^i_0\right\|^{1+m}_{L^{1+m}\left(\R^n\right)}\qquad \forall 1\leq i\leq k.
\end{equation}
\indent By \eqref{eq-estimates-for-u-epsilon-M-similar-to-weak-solution-space} and Young's inequality,
\begin{align}
\int_{0}^{\infty}\int_{\R^n}\left|\nabla \left(w_{\epsilon,M}^{\frac{m-1}{2}}u_{\epsilon,M}^i\right)\right|^2\,dxdt&\leq 2\left(\int_{0}^{\infty}\int_{\R^n}\left(u_{\epsilon,M}^i\left|\nabla w_{\epsilon,M}^{\frac{m-1}{2}}\right|\right)^2\,dxdt+\int_{0}^{\infty}\int_{\R^n}\left(w_{\epsilon,M}^{\frac{m-1}{2}}\left|\nabla u_{\epsilon,M}^i\right|\right)^2\,dxdt\right)\notag\\
&\leq 2\left(\left(\frac{m-1}{m}\right)^2\int_{0}^{\infty}\int_{\R^n}\left|\nabla w_{\epsilon,M}^{\frac{m}{2}}\right|^2\,dxdt+\int_{0}^{\infty}\int_{\R^n}\left(w_{\epsilon,M}^{\frac{m-1}{2}}\left|\nabla u_{\epsilon,M}^i\right|\right)^2\,dxdt\right)\notag\\
&\leq C\sum_{i=1}^k\left\|u^i_0\right\|^{1+m}_{L^{1+m}\left(\R^n\right)}\qquad \forall 1\leq i\leq k\label{eq-estimate-of-gradient-of-bold-w-to-alpha-1-and-u-i-in-L-2}
\end{align}
for some constant $C>0$.\\
By \eqref{eq-estimates-of-L-2-of-u-i-and-H-1-of--sqrt-epsilon-u-i}, 
\begin{equation*}
u^{i}_{\epsilon, M}\in L^{\infty}_t\left(L^2_x\right)\subset L^{2}_{loc}\left(\R^n\times(0,\infty)\right) \qquad \forall 1\leq i\leq k.
\end{equation*}
Thus, $u^{i}_{\epsilon,M}$ converges in $L^{2}_{loc}\left(\R^n\times(0,\infty)\right)$ to some function $u^i$ as $\epsilon\to 0$ and $M\to\infty$. More precisely we can get
\begin{equation}\label{eq-L-1-loc-convergence-of-u-i-and-bold-w}
u^i_{\epsilon,M}\to u^i \quad \mbox{in $L^2_{loc}\cap L^{1+m}_{loc}$} \qquad \mbox{and}\qquad w_{\epsilon,M}\to w=\sum_{i=1}^{k}\left(u^i\right)^2 \quad \mbox{in $L^1_{loc}\cap L^{\frac{1+m}{2}}_{loc}$}
\end{equation}
as $\epsilon\to 0$ and $M\to\infty$.\\
\indent We now are going to show that the function $u^i$ satisfies the weak formula \eqref{eq-identity--of-formula-for-weak-solution}. Since $0\leq u^i_{\epsilon,M}w_{\epsilon,M}^{-\frac{1}{2}}\leq 1$ in $\R^n\times[0,\infty)$ for any $0<\epsilon<1$ and $M>1$, there exists a function $G\in L^{\infty}\left(\R^n\times[0,\infty)\right)$ such that
\begin{equation}\label{eq-convergence-of-ratio-of-u-epsilon-M-and-sqrt-bold-w-epsilon-M}
u^i_{\epsilon,M}w_{\epsilon,M}^{-\frac{1}{2}}\to G \qquad \mbox{in $L^2_{loc}\left(\R^n\times(0,\infty)\right)$} \qquad \mbox{as $\epsilon\to 0$, \,\,$M\to\infty$}.
\end{equation} 
By \eqref{eq-L-1-loc-convergence-of-u-i-and-bold-w}, \eqref{eq-convergence-of-ratio-of-u-epsilon-M-and-sqrt-bold-w-epsilon-M} and H\"older inequality,
\begin{equation}\label{eq-for-equality-between-u-i-and-G-times-sqrt-bold-w}
\begin{aligned}
&\iint_{\mathcal{K}}\left|u^i_{\epsilon,M}-w^{\frac{1}{2}}G\right|^2\,dxdt\\
&\qquad \qquad \leq 2\iint_{\mathcal{K}}w_{\epsilon,M}\left(\frac{u^i_{\epsilon,M}}{w^{\frac{1}{2}}_{\epsilon,M}}-G\right)^2\,dxdt+2\iint_{\mathcal{K}}G^2\left(w_{\epsilon,M}^{\frac{1}{2}}-w^{\frac{1}{2}}\right)^2\,dxdt\\
&\qquad \qquad \leq 2^{2+\frac{2}{1+m}}\left(\iint_{\mathcal{K}}w^{\frac{1+m}{2}}_{\epsilon,M}\,dxdt\right)^{\frac{2}{1+m}}\left(\iint_{\mathcal{K}}\left(\frac{u^i_{\epsilon,M}}{w^{\frac{1}{2}}_{\epsilon,M}}-G\right)^2\,dxdt\right)^{\frac{m-1}{1+m}}\\
&\qquad \qquad \qquad \qquad +2\iint_{\mathcal{K}}\left(w_{\epsilon,M}^{\frac{1}{2}}-w^{\frac{1}{2}}\right)^2\,dxdt\\
&\qquad \qquad \to 0 \qquad \qquad \mbox{as $\epsilon\to 0$ and $M\to \infty$ }
\end{aligned}
\end{equation}
for any compact subset $\mathcal{K}$ of $\R^n\times(0,\infty)$. By \eqref{eq-L-1-loc-convergence-of-u-i-and-bold-w} and \eqref{eq-for-equality-between-u-i-and-G-times-sqrt-bold-w},
\begin{equation}\label{eq-equality-between-u-i-and-G-times-sqrt-bold-w-in-L-2-loc-sense}
u^i\equiv w^{\frac{1}{2}}G \qquad \mbox{in the $L^{2}_{loc}$ sense} \qquad \forall 1\leq i\leq k.
\end{equation}
Let $\vp\in C^{\infty}_0\left(\R^n\times(0,\infty)\right)$ be a test function. Multiplying the first equation of \eqref{eq-for-existence-non-degenerate-pde-with-bold-w-epsilon-M-diffusion-coefficients} by $\vp$ and integrating over $\R^n\times(0,\infty)$, we have
\begin{equation}\label{eq-for-weak-concept-of-solution-type-1}
\begin{aligned}
&m\int_{0}^{\infty}\int_{\R^n}\nabla\left(w_{\epsilon,M}^{\frac{m-1}{2}}u_{\epsilon,M}^i\right)\cdot\nabla\vp\,dxdt-(m-1)\int_{0}^{\infty}\int_{\R^n}u_{\epsilon,M}^iw^{-\frac{1}{2}}_{\epsilon,M}\nabla w_{\epsilon,M}^{\frac{m}{2}}\cdot\nabla\vp\,dxdt\\
&\qquad \qquad +m\epsilon^{\frac{1}{2}}\int_{0}^{\infty}\int_{\R^n}\left(\epsilon^{\frac{1}{2}}\nabla u_{\epsilon,M}^i\right)\cdot\nabla\vp\,dxdt-\int_{0}^{\infty}\int_{\R^n}u^i_{\epsilon,M}\,\vp_t\,dxdt=0.
\end{aligned}
\end{equation}
By \eqref{eq-estimates-for-u-epsilon-M-similar-to-weak-solution-space}, \eqref{eq-estimate-of-gradient-of-bold-w-to-alpha-1-and-u-i-in-L-2} and \eqref{eq-L-1-loc-convergence-of-u-i-and-bold-w}, as $\epsilon\to 0$ and $M\to\infty$,
\begin{equation}\label{eq-convergence-of-nabla-bold-w-alpha-1-and-nabla-w-alpha-u}
\begin{cases}
\begin{aligned}
u_{\epsilon,M}^i&\to u^i \qquad \qquad \mbox{ in $L_{loc}^2\left(\R^n\times(0,\infty)\right)$} \quad \forall 1\leq i\leq k\\
\nabla w_{\epsilon,M}^{\frac{m}{2}}&\to \nabla w^{\frac{m}{2}} \qquad \qquad \mbox{in $L_{loc}^2\left(\R^n\times(0,\infty)\right)$}\\
\nabla\left(w_{\epsilon,M}^{\frac{m-1}{2}}u^i_{\epsilon,M}\right)&\to \nabla\left(w^{\frac{m-1}{2}}u^i\right) \qquad \mbox{in $L_{loc}^2\left(\R^n\times(0,\infty)\right)$} \quad \forall 1\leq i\leq k.
\end{aligned}
\end{cases}
\end{equation}
Letting $\epsilon\to 0$ and then $M\to\infty$ in \eqref{eq-for-weak-concept-of-solution-type-1}, by \eqref{eq-estimates-of-L-2-of-u-i-and-H-1-of--sqrt-epsilon-u-i}, \eqref{eq-convergence-of-ratio-of-u-epsilon-M-and-sqrt-bold-w-epsilon-M} and \eqref{eq-convergence-of-nabla-bold-w-alpha-1-and-nabla-w-alpha-u} $u^i$, $(1\leq i\leq k)$, satisfies
\begin{equation}\label{eq-for-weak-concept-of-solution-type-1-after-limits}
\begin{aligned}
&m\int_{0}^{\infty}\int_{\R^n}\nabla\left(w^{\frac{m-1}{2}}u^i\right)\cdot\nabla\vp\,dxdt\\
&\qquad \qquad-(m-1)\int_{0}^{\infty}\int_{\R^n}G\,\nabla w^{\frac{m}{2}}\cdot\nabla\vp\,dxdt-\int_{0}^{\infty}\int_{\R^n}u^i\vp_t\,dxdt=0. \qquad \forall 1\leq i\leq k.
\end{aligned}
\end{equation}
By \eqref{eq-equality-between-u-i-and-G-times-sqrt-bold-w-in-L-2-loc-sense} and \eqref{eq-for-weak-concept-of-solution-type-1-after-limits}, we can get
\begin{equation*}\label{eq-for-weak-concept-of-solution-type-1-after-limits-0}
m\int_{0}^{\infty}\int_{\R^n}w^{\frac{m-1}{2}}\nabla u^i\cdot\nabla\vp\,dxdt -\int_{0}^{\infty}\int_{\R^n}u^i\vp_t\,dxdt=0.\qquad \forall 1\leq i\leq k .
\end{equation*}
To complete the proof, we now are going to show that 
\begin{equation}\label{eq-converges-of-u-to-initial-data-as-t-to-zero}
u^i\left(\cdot,t\right)\to u^i_0\qquad  \mbox{in $L^1$ as $t\to 0^+$},\,\,\forall 1\leq i\leq k.
\end{equation} 
Let $\eta(x)\in C_0^{2}\left(\R^n\right)$. Let $0<t<1$. Multiply the first equation of \eqref{eq-for-existence-non-degenerate-pde-with-bold-w-epsilon-M-diffusion-coefficients} by $\eta$ and integrate it over $\R^n\times(0,t)$. Then, by an argument similar to the proof of Lemma \ref{lem-space-where-weak-solution-belongs-to-1} we have
\begin{equation}\label{compare-between-u-and-u-0-with-eta23579}
\begin{aligned}
&\left|\int_{\R^n}u_{\epsilon,M}(x,t)\eta(x)\,dx-\int_{\R^n}u_{0,\epsilon,M}(x)\eta(x)\,dx\right|\\
&\qquad \qquad  \leq \int_{0}^{t}\int_{\R^n}w_{\epsilon,M}^{\frac{m-1}{2}}(x,t)\left|\nabla u_{\epsilon,M}(x,t)\right|\left|\nabla\eta(x)\right|\,dxdt\\
&\qquad \qquad \leq C\left(\left\|\bold{u}_0\right\|_{L^{1+m}\left(\R^n\right)},\left\|\nabla\eta\right\|_{L^{\infty}}\right)\sqrt{t},\qquad \forall 0<t<1.
\end{aligned}
\end{equation}
Thus, letting $M\to\infty$, $\epsilon\to 0$ and then $t\to 0$ in \eqref{compare-between-u-and-u-0-with-eta23579}, the claim follows. Therefore, $u$ is a weak solution of \eqref{eq-main-equation-of-system} and the lemma follows.
\end{proof}

\subsection{Local H\"older Continuity.}
As the second step of this section, we are going to give an explanation about the local h\"older  continuity of solution $\bold{u}$ of \eqref{eq-main-equation-of-system}. Let $(x_0,t_0)\in\R^n\times(0,\infty)$ and consider the cylinder
\begin{equation*}\label{eq-local-set-for-local-contiunity-with-radius-R}
\left(x_0,t_0\right)+Q\left(R_0,R_0^{2-\epsilon}\right)=\left(x_0,t_0\right)+B_{R}\times\left(-R_0^{2-\epsilon},0\right)\subset \R^n\times(0,\infty),\qquad \left(0<R_0\leq 1\right)
\end{equation*}
where $\epsilon>0$ is a sufficiently small number and $B_R$ is the ball centered at $x=0$ of radius $R>0$. For the simplicity, we may assume without loss of generality that $(x_0,t_0)=(0,0)$.\\
\indent By Lemma \ref{eq-lemma-L-infty-bound-of-bold-u-absolute-value}, there exists uniform number $\Lambda>0$ such that
\begin{equation*}
u^i\leq \left|\bold{u}\right|\leq \Lambda<\infty \qquad \forall (x,t)\in Q\left(R_0,R_0^{2-\epsilon}\right),\,\,1\leq i\leq k
\end{equation*}
Hence, for each $1\leq i\leq k$ we can define
\begin{equation*}\label{eq-definition-mu-+-and-mu---and-omega-for-is}
\left(\mu^i\right)^+=\esssup_{Q(R_0,R_0^{2-\epsilon})} u^i, \qquad \left(\mu^i\right)^-=\essinf_{Q(R_0,R_0^{2-\epsilon})} u^i, \qquad \omega^i=\osc_{Q(R_0,R_0^{2-\epsilon})} u^i=\left(\mu^i\right)^+-\left(\mu^i\right)^-. 
\end{equation*}
Since $\left|\bold{u}\right|\geq u^i$ for all $1\leq i\leq k$, the equation $\left(u^i\right)_t=\nabla\left(\left|\bold{u}\right|^{m-1}\nabla u^i\right)$ is uniformly parabolic in $Q\left(R_0,R_0^{2-\epsilon}\right)$ if $\left(\mu^i\right)^->0$ for some $1\leq i\leq k$. Then, by standard regularity theory for the parabolic equation \cite{LSU} local H\"older continuity of solution follows. Hence, we can assume that 
\begin{equation*}
\left(\mu^i\right)^-=0 \qquad \forall 1\leq i\leq k.
\end{equation*}
Moreover, if $\left(\mu^{i}\right)^+=0$ for some $1\leq i\leq k$ then
\begin{equation*}
u^{i}\equiv 0 \qquad \mbox{on $Q\left(R_0,R_0^{2-\epsilon}\right)$}.
\end{equation*}
This means that we don't need to consider the effect of $u^{i}$ in studying the evolution of solution $\bold{u}$ in that region. Hence, we also assume that
\begin{equation}\label{eq-non-degeneracy-of-omega-i-s-34}
\omega^{\,i}=\left(\mu^i\right)^+>0 \qquad \forall 1\leq i\leq k.
\end{equation}
Let 
\begin{equation*}
\omega_{_M}=\max_{1\leq i\leq k}\omega^{\,i}
\end{equation*}
and construct the cylinder
\begin{equation}\label{eq-construction-of-cylinder-with-some-constant-A-which-is-bigger}
Q\left(R,\theta^{-\alpha_0}R^2\right)=B_{R}\times\left(-\theta^{-\alpha_0}R^2,0\right) \qquad \left(\theta=\frac{\omega_{_M}}{4},\,\,\alpha_0=m-1\right).
\end{equation}
Assume that the radius $0<R<R_0$ is sufficiently small that
\begin{equation}\label{eq-condition-between-R-and-theta-alpha--1}
\theta^{\alpha_{0}}>R^{\epsilon}.
\end{equation}
By \eqref{eq-construction-of-cylinder-with-some-constant-A-which-is-bigger} and \eqref{eq-condition-between-R-and-theta-alpha--1}, 
\begin{equation*}
Q\left(R,\theta^{-\alpha_{0}}R^2\right)\subset Q\left(R,R^{2-\epsilon}\right) \subset Q\left(R_0,R_0^{2-\epsilon}\right)  \qquad \forall 1\leq i\leq k.
\end{equation*}
and
\begin{equation*}
\osc_{Q\left(R,\theta^{-\alpha_{0}}R^2\right)}u^i\leq \omega_{_{M}}=4\theta, \qquad \forall 1\leq i\leq k 
\end{equation*}
and
\begin{equation}\label{eq-upper-bound-of-absolute-of-bold-u-by-omega-M-and-k}
\left|\bold{u}\right|\leq 4\sqrt{k}\,\theta \qquad \mbox{in $Q\left(R,\theta^{-\alpha_{0}}R^2\right)$}.
\end{equation}
Then, by arguments similar to the proofs of the Lemma 3.2 and Lemma 3.11 of \cite{KL1} we can get the following two alternatives.
\begin{lemma}\label{lem-the-first-alternative-for-holder-estimates}
Let $\bold{u}=\left(u^1,\cdots,u^k\right)$ be a weak solution of \eqref{eq-main-equation-of-system}. Then there exists a number $\rho_0>0$ such that if
\begin{equation}\label{eq-first-condition-of-small-region-of-lower-for-holder-estimates}
\left|\left\{\left(x,t\right)\in Q\left(R,\theta^{-\alpha_{0}}R^2\right):u^i(x,t)<\frac{\omega_{_M}}{2}\right\}\right|\leq\rho_0\left|Q\left(R,\theta^{-\alpha_{0}}R^2\right)\right|
\end{equation}
then,
\begin{equation*}
u^i(x,t)>\frac{\omega_{_M}}{4} \qquad \mbox{for all $(x,t)\in Q\left(\frac{R}{2},\theta^{-\alpha_{0}}\left(\frac{R}{2}\right)^2\right)$}.
\end{equation*}
\end{lemma}
\begin{lemma}\label{lem-the-second-alternative-for-holder-estimates}
Let $\bold{u}=\left(u^1,\cdots,u^k\right)$ be a weak solution of \eqref{eq-main-equation-of-system}. If \eqref{eq-first-condition-of-small-region-of-lower-for-holder-estimates} is violated, then the number $s^{\ast}\in\N$ can be chosen such that
\begin{equation*}
u^i(x,t)\leq \left(1-\frac{1}{2^{s^{\ast}+1}}\right)\omega_{_M} \qquad \mbox{a.e. on $Q\left(\frac{R}{2},\frac{\rho_0}{2}\theta^{-\alpha_{0}}\left(\frac{R}{2}\right)^2\right)$}. 
\end{equation*}
\end{lemma}
As mentioned in Section 3 of \cite{KL1}, the constants $\rho_0$ and $s^{\ast}$ is independent of $\omega_1$, $\cdots$, $\omega_k$ and $R$. Thus, by Lemma \ref{lem-the-first-alternative-for-holder-estimates} and Lemma \ref{lem-the-second-alternative-for-holder-estimates} oscillations of all $u^i$, $\left(1\leq i\leq k\right)$ are decreasing in the same proportion, i.e., we can have the following \textbf{Oscillation Lemma}.
\begin{lemma}\label{lem-Oscillation-Lemma} 
There exist numbers $\rho_0$, $\sigma_0\in\left(0,1\right)$ depending only on the $m$, $k$ and $n$ such that 
\begin{equation}\label{eq-inequality-for-Oscillation-Lemma}
\osc_{Q\left(\frac{R}{2},\frac{\rho_0}{2}\theta^{-\alpha_{0}}\left(\frac{R}{2}\right)^{2}\right)}u^i=\sigma_0\omega_{_M}\qquad \forall 1\leq i\leq k.
\end{equation}
\end{lemma}
Let 
\begin{equation*}
R_1=\frac{R}{\sigma_0^{\frac{\alpha_0}{2}}C}=\frac{R}{\overline{C}}
\end{equation*}
where $C$ is a constant satisfying
\begin{equation*}
C\geq \max\left(\frac{2}{\sqrt{\sigma_0^{\alpha_0}}}\,, \,2\sqrt{\frac{2}{\rho_0\sigma_0^{\alpha_0}}},\frac{1}{\sigma_0^{1+\frac{\alpha_0}{2}}}\right).
\end{equation*}
Then
\begin{equation*}
Q\left(R_1,\theta^{-\alpha_0}R_1^2\right) \subset Q\left(\frac{R_0}{C},\left(\sigma_0\theta\right)^{-\alpha_0}\left(\frac{R_0}{C}\right)^2\right) \subset Q\left(\frac{R_0}{2},\frac{\rho_0}{2}\theta^{-\alpha_0}\left(\frac{R_0}{2}\right)^2\right).
\end{equation*}
Thus we have
\begin{equation*}
\osc_{Q\left(R_1,\theta^{-\alpha_0}R_1^2\right)}u^i\leq \sigma_0\omega_{_M}=4\sigma_0\theta, \qquad \forall 1\leq i\leq k.
\end{equation*}
and
\begin{equation*}
\left|\bold{u}\right|\leq \sqrt{k}\,\sigma_0\omega_{_M}=4\sqrt{k}\,\sigma_0\theta \qquad \mbox{on $Q\left(R_1,\theta^{-\alpha_0}R_1^2\right)$}.
\end{equation*}
Applying the \textbf{Oscillation Lemma} again with $\theta$, $\sqrt{k}\theta$ being replaced by $\sigma_0\theta$, $\sqrt{k}\sigma_0\theta$, respectively, we can get
\begin{equation}\label{eq-inequality-for-Holder-continuity-first-step}
\osc_{Q\left(\frac{R}{\overline{C}^2},\theta^{-\alpha_0}\left(\frac{R}{\overline{C}^2}\right)^2\right)}u=\osc_{Q\left(\frac{R_1}{\sigma_0^{\frac{\alpha_0}{2}}C},\left(\sigma_0^{2}\theta\right)^{-\alpha_0}\left(\frac{R_1}{C}\right)^2\right)}u\leq \osc_{Q\left(\frac{R_1}{2},\frac{\rho_0}{2}\left(\sigma_0\theta\right)^{-\alpha_{0}}\left(\frac{R_1}{2}\right)^{2}\right)}u=\sigma_0\omega_1=\sigma_0^2\omega_0.
\end{equation}
Continuing this process, we can have 
\begin{equation*}
\osc_{Q\left(\frac{R}{\overline{C}^j},\theta^{-\alpha_0}\left(\frac{R}{\overline{C}^j}\right)^2\right)}u^i\leq\sigma_0^j\omega_0 \qquad \forall j\in\N,\,\,1\leq i\leq k.
\end{equation*}
By an argument similar to the proof of Theorem 3.12 of \cite{KL3}, 
\begin{equation*}
\osc_{Q\left(r,\theta^{-\alpha_0}r^2\right)}u^i\leq K\omega_{_M}\left(\frac{r}{R}\right)^{\beta} \qquad \left(0<r\leq R\right),\,\,1\leq i\leq k
\end{equation*}
holds for $0<\beta=-\log_{\overline{C}}\sigma_0<1$, $K=\frac{1}{\sigma_0}$. Therefore local H\"older continuity of solution $\bold{u}$ follows.
\begin{thm}\label{eq-local-continuity-of-solution}
Let $\bold{u}=\left(u^1,\cdots,u^k\right)$ be solution of \eqref{eq-main-equation-of-system} with initial data $u^i_0\in L^1\left(\R^n\right)\cap L^{1+m}\left(\R^n\right)$. Then the component $u^i$, $\left(1\leq i\leq k\right)$, is locally H\"older continuous in $\R^n\times\left(0,\infty\right)$. Especially, all components have the same modulus of continuity.
\end{thm}

\section{Evolution of population densities of $k$-species: Isolation, Synchronization and Stabilization}
\setcounter{equation}{0}
\setcounter{thm}{0}

\subsection{Isolated Species: Before Mixing of Isolated Supports}
In this subsection, we will investigate that there exists a waiting time until any two species interact with each other under the assumption that all species are located separately from each other initially.
\begin{proof}[\textbf{Proof of Theorem \ref{thm-existence-of-waiting-time-if-initially-separated}}]
Let $1\leq i\leq k$. Since 
\begin{equation*}
\overline{\textbf{supp}\,u^i_0}\subset \bigcup_{x\in\,\overline{\textbf{supp}\,u^i_0}}B_{\frac{d}{8}}(x)
\end{equation*} 
and $\textbf{supp}\,u^i_0$ is compact, there exist $x^i_1$, $\cdots$, $x^i_{l_i}$ such that
\begin{equation*}
\overline{\textbf{supp}\,u^i_0}\subset \bigcup_{j=1}^{l_i}B_{\frac{d}{8}}\left(x_{j}^i\right).
\end{equation*}
By Lemma \ref{eq-lem-boundedness-of-support-of-solution}, the functions $u^i$, $\cdots$, $u^k$ have at most finite speeds of propagations, i.e., there are times $T_1$, $\cdots$, $T_k>0$ such that
\begin{equation*}
\overline{\textbf{supp}\,u^i\left(\cdot,t\right)}\subset \bigcup_{j=1}^{l_i}B_{\frac{d}{4}}\left(x_{j}^i\right).\qquad \forall 1\leq i\leq k,\,\,0\leq t\leq T_i.
\end{equation*}
By \eqref{eq-infimum-distance-between-supports-of-u-i-s},
\begin{equation*}
\textbf{dist}\,\left(\bigcup_{j=1}^{l_{i_1}}B_{\frac{d}{4}}\left(x_{j}^{i_1}\right),\,\bigcup_{j=1}^{l_{i_2}}B_{\frac{d}{4}}\left(x_{j}^{i_2}\right)\right)\geq \frac{d}{2} \qquad \forall 1\leq i_1<i_2\leq k.
\end{equation*}
Thus, \eqref{eq-waiting-time-of-mixing-of-isolated-supports} holds for $T=\min\left\{T_1,\cdots,T_k\right\}$ and the theorem follows.
\end{proof}

\subsection{Synchronization: Spontaneous Mixing of Supports.}\label{subsection-Asymptotic-behaviour-after-mixing-of-supports}
This subsection is devoted to explain how fast two different species diffuse in each other's territory when they begin to mix.\\
\indent  Let $\bold{u}=\left(u^1,\cdots,u^k\right)$ be continuous solutions of \eqref{eq-main-equation-of-system} and suppose that $u^i(x_0,t_0)>0$ for some points  $\left(x_0,t_0\right)\in\R^n\times(0,\infty)$ and for some integer $1\leq i\leq k$. Then the diffusion coefficients of \eqref{eq-main-equation-of-system} is uniformly parabolic near the point $\left(x_0,t_0\right)$. By standard theory for the non-degenerate parabolic equation \cite{LSU}, the region near the point $\left(x_0,t_0\right)$ where the function $u^i$, $\left(1\leq i\leq k\right)$, is strictly positive can be immediately extended to the whole space where $\left|\bold{u}\right|$ is strictly positive. With this observation, we now are going to give a proof of the Theorem \ref{eq-lem-equailty-of-support-ofu-i-and-u-j-when-initially-joined}.

\begin{proof}[\textbf{Proof of Theorem \ref{eq-lem-equailty-of-support-ofu-i-and-u-j-when-initially-joined}}]
The proof is almost the same as that of Lemma 2.7 of \cite{KL1}. Thus, we will just use a modification of their proof here. If \eqref{eq-statements-of-equality-of-supprot-of-u-i-s-forall-i} is violated for some $t_1>t_0$, then there exist a point $x_1\in\R^n$ such that
\begin{equation*}
\left|\bold{u}\right|\left(x_1,t_1\right)>0 \qquad \mbox{and} \qquad x_1\in\partial\,\textbf{supp}\,u^{\,j}(t_1)\cup \partial\,\textbf{supp}\,u^{\,l}(t_1).
\end{equation*}
Without loss of generality, we may assume that
\begin{equation}\label{eqw-condition-of-suppose-not-strictly-posistive-on-bold-u=andon=a-boundary-of-u-i-=2}
\left|\bold{u}\right|\left(x_1,t_1\right)>0 \qquad \mbox{and} \qquad x_1\in\partial\,\textbf{supp}\,u^{\,j}(t_1).
\end{equation}
Since $u^i$ are continuous for all $1\leq i\leq k$, there exists a sufficiently small number $0<\epsilon<\frac{1}{2}\left(t_1-t_0\right)$ such that
\begin{equation*}
B_{\epsilon}(x_1)\subset\textbf{supp}\,\left|\bold{u}\right|(s) \qquad \forall t_1-\epsilon\leq s\leq t_1
\end{equation*}
and
\begin{equation}\label{eq-not-identically-zero-of-u-i-1-on-some-region}
u^{\,j}\left(x,t\right)\not\equiv 0 \qquad \forall(x,t)\in B_{\epsilon}(x_1)\times\left[t_1-\epsilon,t_1\right]
\end{equation}
\indent Let $v_{0}(x)=u^{\,j}\left(x,t_1-\frac{1}{2}\epsilon\right)\chi_{B_{\epsilon}\left(x_1\right)}$. Since $\left|\bold{u}\right|$ is strictly positive on $B_{\epsilon}(x_1)\times\left[t_1-\epsilon,t_1\right]$, by standard theory for the non-degenerate parabolic equation \cite{LSU} there exists an unique solution $v$ of 
\begin{equation*}
\begin{cases}
\begin{array}{rlcl}
v_t(x,t)&=\nabla\cdot\left(\left|\bold{u}\right|^{m-1}\left(x,t+t_1-\epsilon\right)\nabla v\left(x,t\right)\right)&\quad& \mbox{in $B_{\epsilon}(x_1)\times\left(0,\frac{\epsilon}{2}\right]$}\\
v(x,t)&=0 &\quad&\mbox{on $\partial B_{\epsilon}(x_1)\times\left(0,\frac{\epsilon}{2}\right]$}\\
v(x,0)&=v_0(x) &\quad& \mbox{in $B_{\epsilon}(x_1)$}.
\end{array}
\end{cases}
\end{equation*}
Moreover, by \eqref{eq-not-identically-zero-of-u-i-1-on-some-region} there exists a constant $c_1>0$ such that
\begin{equation}\label{eq-strictly-positive-of-v-on-B-half-epsilon-at-x-1-half-epsilon}
v_0(x)\not\equiv 0 \quad \mbox{in $B_{\epsilon}(x_1)$}\qquad\Rightarrow \qquad v\left(x,\frac{1}{2}\epsilon\right)\geq c_1>0 \qquad \forall x\in B_{\frac{\epsilon}{2}}(x_1).
\end{equation}
Since $u^{\,j}\left(x,t+t_1-\frac{1}{2}\epsilon\right)$ and $v(x,t)$ satisfies the same equation in $B_{\epsilon}(x_1)\times\left(0,\frac{\epsilon}{2}\right]$ and $u^{\,j}\left(x,t+t_1-\frac{1}{2}\epsilon\right)\geq v(x,t)$ on the parabolic boundary of $B_{\epsilon}(x_1)\times\left(0,\frac{\epsilon}{2}\right]$, by an argument similar to the proof of Lemma \ref{lem-uniquness-of-weak-solution-2} we have
\begin{equation}\label{eq-comparision-between-u-i-1-at-t-t-1-half-epsilon-and-v}
u^{\,j}\left(x,t+t_1-\frac{1}{2}\epsilon\right)\geq v(x,t) \qquad \forall (x,t)\in B_{\epsilon}(x_1)\times\left(0,\frac{\epsilon}{2}\right].
\end{equation}
By \eqref{eq-strictly-positive-of-v-on-B-half-epsilon-at-x-1-half-epsilon} and \eqref{eq-comparision-between-u-i-1-at-t-t-1-half-epsilon-and-v},
\begin{equation*}
u^{\,j}\left(x_1,t_1\right)\geq v\left(x_1,\frac{\epsilon}{2}\right)\geq c_1>0
\end{equation*}
, which contradicts \eqref{eqw-condition-of-suppose-not-strictly-posistive-on-bold-u=andon=a-boundary-of-u-i-=2}. Thus \eqref{eq-statements-of-equality-of-supprot-of-u-i-s-forall-i} holds and the lemma follows.
\end{proof}

\subsection{Stabilization: Improvements of Similarities Between Species.}

In this section, we will consider the asymptotic large time behaviour of functions $u^i$, $\left(1\leq i\leq k\right)$. More precisely, this section is devoted to study the uniform convergence between $u^i$, $\left(1\leq i\leq k\right)$, and Barenblatt solution of the porous medium equation.

\subsubsection{Uniqueness}
\begin{lemma}\label{lem-first-step-for-uniqueness-u-i-expressed-by-v-and-its-volume}
Let $m>1$ and let $M_1$, $\cdots$, $M_k$ be positive real numbers. Let $\delta$ is Dirac's delta function. For each $1\leq i\leq k$, let $u^i$ be a solution of \eqref{eq-main-equation-of-system} with initial data $u_0^i=M_i\delta$. Then
\begin{equation*}
M_ju^i(x,t)=M_iu^j(x,t) \qquad \forall (x,t)\in\R^n\times(0,\infty),\,\,1\leq i,\,j\leq k,
\end{equation*}
i.e., 
\begin{equation}\label{eq-expression-of-solution-u-to-i-with-barenblatt-initial-trace}
u^i=M_iv =M_i\left(\frac{\sum_{i=1}^ku^i}{\sum_{i=1}^kM_i}\right).
\end{equation}
\end{lemma}
\begin{proof}
For any $1\leq i$, $j\leq k$, set $w=M_ju^i-M_iu^j$. Then
\begin{equation*}
w(x,0)=0 \qquad \mbox{in $\R^n$}
\end{equation*}
and
\begin{equation*}
\begin{aligned}
w_t=M_j\left(u^i\right)_t-M_i\left(u^j\right)_t=\nabla\cdot\left(m\left|\bold{u}\right|^{m-1}\nabla w\right) \qquad \mbox{in $\R^n\times(0,\infty)$}.
\end{aligned}
\end{equation*}
Then by an argument similar to the proof of Lemma \ref{lem-uniquness-of-weak-solution-2}, we have
\begin{equation}\label{eq-equivalent-between-bro-u-to-i-s}
\begin{aligned}
w(x,t)\equiv 0 \qquad \Rightarrow \qquad M_ju^i(x,t)=M_iu^{j}(x,t)\qquad \forall (x,t)\in\R^n\times\left[0,\infty\right)
\end{aligned}
\end{equation}
Moreover by \eqref{eq-equivalent-between-bro-u-to-i-s},
\begin{equation*}
\left(\sum_{j=1}^kM_j\right)u^i(x,t)=\sum_{j=1}^kM_ju^i(x,t)=\sum_{j=1}^kM_iu^j(x,t)=M_i\left(\sum_{j=1}^ku^j(x,t)\right) \qquad \forall (x,t)\in\R^n\times\left[0,\infty\right).
\end{equation*}
Thus \eqref{eq-expression-of-solution-u-to-i-with-barenblatt-initial-trace} holds and the lemma follows.
\end{proof}

Let $v$ be the function given in Lemma \ref{lem-first-step-for-uniqueness-u-i-expressed-by-v-and-its-volume}. By direct computation, $v$ satisfies
\begin{equation*}
v_t=\left|\bold{M}\right|^{m-1}\nabla\cdot\left(m\left|v\right|^{m-1}\nabla v\right) \qquad \mbox{in $\R^n\times(0,\infty)$}
\end{equation*}
where $\bold{M}=\left(M_1,\cdots, M_k\right)$ and
\begin{equation*}
v(x,0)=\delta(x) 
\end{equation*}
where $\delta$ is Dirac's delta function. Then by M. Pierre's uniqueness theorem \cite{Pi}, we can have the following result.
\begin{thm} [cf.Theorem 13.6 of \cite{Va1}]\label{thm-convergence-of-v-defined-above-to-some-mass-times-barrenblatt-function}
Let $m>1$ and let $M_1$, $\cdots$, $M_k$ be positive real numbers. Let $\delta$ is Dirac's delta function. Then, for each $1\leq i\leq k$ there exists an unique solution $u^i$ of \eqref{eq-main-equation-of-system} with initial data $u_0^i=M_i\delta$. Here the solution $u^i$ is given by
\begin{equation*}
u^i(x,t)=\frac{M_i}{\left|\bold{M}\right|}\mathcal{B}_{\left|\bold{M}\right|}(x,t) \qquad \forall (x,t)\in \R^n\times\left[0,\infty\right)
\end{equation*}
where $\bold{M}=\left(M_1,\cdots,M_k\right)$ and the function $\mathcal{B}_{M}$ is the self-similar Barenblatt solution of the porous medium equation with $L^1$ mass $M$.
\end{thm}

\subsubsection{Scaling and Uniform estimates}

For each $1\leq i\leq k$, let $u^i$ be the function in \eqref{eq-main-equation-of-system} with $L^1$ mass $M_i$. For $\lambda>0$, construct the families of functions
\begin{equation}\label{eq-def-of-rescaled-function-u-lambda-i}
u_{\lambda}^i(x,t)=\lambda^{a_1}u^i\left(\lambda^{a_2}x,\lambda t\right) \qquad \forall 1\leq i\leq k 
\end{equation}
and
\begin{equation}\label{eq-def-of-rescaled-function-bold-w-lambda-i}
w_{\lambda}(x,t)=\left|\bold{u}_{\lambda}\right|^2(x,t) \qquad \qquad \left(\,\bold{u}_{\lambda}=(u_{\lambda}^1,\cdots,u_{\lambda}^k)\,\right)
\end{equation}
where $a_1$ and $a_2$ are given by \eqref{eq-constant-alpha-1-beta-1-k}. Then $u_{\lambda}^i$ satisfies the equation
\begin{equation}\label{problem-of-after-scaling-computation-for-asymptotic}
\begin{cases}
\begin{aligned}
\left(u^i_{\lambda}\right)_t=\nabla\cdot\left(m\,w_{\lambda}^{\frac{m-1}{2}}\nabla u^i_{\lambda}\right) \qquad \mbox{in $\R^n\times(0,\infty)$}\\
u^i_{\lambda}(x,0)=u_0^i\left(\lambda^{a_2}x\right):=u^i_{0,\lambda}(x) \qquad \forall x\in\R^n.
\end{aligned}
\end{cases}
\end{equation}
By Lemma \ref{lem-conservation-law-of-L-1-mass},
\begin{equation}\label{eq-mass-conservation-of-rescaled-solution-u-i-lambda}
\int_{\R^n}u_{\lambda}^i(x,t)\,dx=\int_{\R^n}\lambda^{a_1}u^i\left(\lambda^{a_2}x,\lambda t\right)\,dx=\int_{\R^n}u^i\left(y,\lambda t\right)\,dy=M_i<\infty \qquad \forall \lambda>0, \,\,t\geq 0,\,\,1\leq i\leq k.
\end{equation}
Hence the families $\left\{u^i_{\lambda}\right\}_{\lambda\geq 1}$, $\left(i=1,\cdots,k\right)$, is uniformly bounded in $L^1\left(\R^n\right)$ for all $t>0$. By Lemma \ref{eq-lemma-L-infty-bound-of-bold-u-absolute-value},
\begin{equation*}
\left\|u_{\lambda}^i\left(\cdot,1\right)\right\|_{L^{\infty}}\leq \left\|\,\left|\bold{u}_{\lambda}\right|\left(\cdot,1\right)\right\|_{L^{\infty}}=\lambda^{a_1}\left\|\,\left|\bold{u}\right|\left(\cdot,\lambda\right)\right\|_{L^{\infty}}\leq \lambda^{a_1}\frac{C\left\|\,\left|\bold{u}_0\right|\right\|^{\frac{2a_1}{n}}_{L^1}}{\lambda^{a_1}}\leq CM^{\frac{2a_1}{n}} \qquad \forall 1\leq i\leq k
\end{equation*}
where $M=M_1+\cdots+M_k$ is independent to $\lambda>0$. Similarly
\begin{equation}\label{eq-L-infty-bound-of-rescaled-solution-u-i-lambda-at-t-0}
\left\|u^i_{\lambda}\left(\cdot,t_0\right)\right\|_{L^{\infty}}\leq CM^{\frac{2a_1}{n}}t_0^{-a_1} \qquad \forall t_0>0,\,\,1\leq i\leq k.
\end{equation}
By \eqref{eq-mass-conservation-of-rescaled-solution-u-i-lambda} and \eqref{eq-L-infty-bound-of-rescaled-solution-u-i-lambda-at-t-0} and Interpolation theory,
\begin{equation*}
\mbox{$\left\|u_{\lambda}^i(\cdot,t_0)\right\|_{L^p}$ is equibounded for all $p\in[1,\infty]$}.
\end{equation*}
By an arguments similar to the proof of Theorem \ref{thm-existence-of-weak-solution-u-i-of-system}, we can get estimates for the diffusion coefficients $w_{\lambda}$, i.e., there exists an uniform constant $C>0$ such that
\begin{equation}\label{eq-uniform-estimates-of-bold-w-to-some-power-and-multiply-u-to-i}
\left\|w_{\lambda}^{\frac{m-1}{2}}u_{\lambda}^i\right\|_{L^2\left(0,\infty;H^1(\R^n)\right)},\quad \left\|w_{\lambda}^{\frac{m}{2}}\right\|_{L^2\left(0,\infty;H^1(\R^n)\right)}\leq C \qquad \forall 1\leq i\leq k.
\end{equation}
\subsubsection{Limit function of solution $u^i$, $1\leq i\leq k$.}

We will use a modification of techniques of the proofs in the Section 4 of \cite{KL1} and proofs in the Section 18 of \cite{Va1} to prove the asymptotic behaviour of solution of degenerated parabolic system.

\begin{proof}[\textbf{Proof of Theorem \ref{thm-asymptotic-behaviour-of-system-PME-L-1=and=L=infty}}]
We will use a modification of the technique of proof of Theorem 1.2 of \cite{KL1} to prove theorem. For any $\lambda>0$, let $u_{\lambda}^i$ and $w_{\lambda}$ be given by \eqref{eq-def-of-rescaled-function-u-lambda-i} and \eqref{eq-def-of-rescaled-function-bold-w-lambda-i}, respectively. By \eqref{eq-L-infty-bound-of-rescaled-solution-u-i-lambda-at-t-0}, the family $\left\{u_{\lambda}^i\right\}_{\lambda>0}$ is uniformly bounded in $\R^n\times\left[t_0,\infty\right)$ for any $t_0>0$. This implies that $\left\{u_{\lambda}^i\right\}_{\lambda>0}$ is relatively compact in $L^1_{loc}\left(\R^n\times(0,\infty)\right)$. Thus for any sequence $\lambda_n\to\infty$ as $n\to\infty$, the sequence $\left\{u_{\lambda_n}^i\right\}_{n\in\N}$ has a subsequence which we may assume without loss of generality to be the sequence itself that converges in $L^1_{loc}\left(\R^n\times(0,\infty)\right)$ to some function $u^i_{\infty}$ in $\R^n\times(0,\infty)$ as $n\to\infty$. By the definition of $w_{\lambda}$, we also have
\begin{equation*}
w_{\lambda_n}\to w_{\infty}=\left|\bold{u}_{\infty}\right|^2=\sum_{i=1}^{k}\left(u_{\infty}^i\right)^2 \qquad \mbox{as $n\to\infty$}. 
\end{equation*}
\indent Let $\vp\in C_0^{\infty}\left(\R^n\times(0,\infty)\right)$ be a test function. Multiplying the first equation in \eqref{problem-of-after-scaling-computation-for-asymptotic} by $\vp$ and integrating over $\R^n\times(0,\infty)$, we have
\begin{equation}\label{eq-weak-formula-before-letting-lambda-to-infty-in-aysmptotic}
\begin{aligned}
&m\int_{0}^{\infty}\int_{\R^n}w_{\lambda}^{\frac{m-1}{2}}\nabla u_{\lambda}^i\cdot\nabla\vp\,dxdt-\int_{0}^{\infty}\int_{\R^n}u_{\lambda}^i\,\vp_t\,dxdt\\
&\qquad \qquad =m\int_{0}^{\infty}\int_{\R^n}\nabla\left(w_{\lambda}^{\frac{m-1}{2}}\,u_{\lambda}^i\right)\cdot\nabla\vp\,dxdt-\left(m-1\right)\int_{0}^{\infty}\int_{\R^n}\left(u_{\lambda}^i\,w_{\lambda}^{-\frac{1}{2}}\right)\,\left(\nabla w_{\lambda}^{\frac{m}{2}}\cdot\nabla\vp\right)\,dxdt\\
&\qquad\qquad \qquad \qquad-\int_{0}^{\infty}\int_{\R^n}u_{\lambda}^i\,\vp_t\,dxdt=0 \qquad \qquad \forall 1\leq i\leq k.
\end{aligned}
\end{equation}
By the convergence of $u_{\lambda}^i$ and \eqref{eq-uniform-estimates-of-bold-w-to-some-power-and-multiply-u-to-i}, 
\begin{equation}\label{eq-convergence-of-nabla-bold-w-lambda-alpha-1-and-nabla-w-lambda-alpha-u}
\begin{cases}
\begin{aligned}
u_{\lambda}^i&\to u_{\infty}^i \qquad \qquad \mbox{in $L^1\left(\R^n\times(0,\infty)\right)$} \quad \forall 1\leq i\leq k\\
\nabla w_{\lambda}^{\frac{m}{2}}&\to \nabla w_{\infty}^{\frac{m}{2}} \qquad \qquad \mbox{in $L_{loc}^2\left(\R^n\times(0,\infty)\right)$}\\
\nabla\left(w_{\lambda}^{\frac{m-1}{2}}u^i_{\lambda}\right)&\to \nabla\left(w_{\infty}^{\frac{m-1}{2}}u_{\infty}^i\right) \qquad \mbox{in $L_{loc}^2\left(\R^n\times(0,\infty)\right)$} \quad \forall 1\leq i\leq k
\end{aligned}
\end{cases}
\end{equation}
as $\lambda\to\infty$. Moreover, by an argument similar to the proof of \eqref{eq-equality-between-u-i-and-G-times-sqrt-bold-w-in-L-2-loc-sense} we also have
\begin{equation}\label{eq-convergence-of-nabla-bold-w-lambda-minus-power-and-nabla-u-i-lambda}
u^i_{\lambda}w^{-\frac{1}{2}}_{\lambda}\to G \qquad \mbox{in $L^2_{loc}$} \qquad \mbox{as $\lambda\to\infty$}\qquad \forall 1\leq i\leq k
\end{equation}
for some function $G$ in $L^2_{loc}$ satisfying $u^i_{\infty}=w^{\frac{1}{2}}_{\infty}G$. Letting $\lambda\to\infty$ in \eqref{eq-weak-formula-before-letting-lambda-to-infty-in-aysmptotic}, by \eqref{eq-convergence-of-nabla-bold-w-lambda-alpha-1-and-nabla-w-lambda-alpha-u} and \eqref{eq-convergence-of-nabla-bold-w-lambda-minus-power-and-nabla-u-i-lambda} we have
\begin{equation}\label{eq-weak-formula-of-u-i-infty-10}
m\int_{0}^{\infty}\int_{\R^n}w_{\infty}^{\frac{m-1}{2}}\nabla u_{\infty}^i\cdot\nabla\vp\,dxdt-\int_{0}^{\infty}\int_{\R^n}u_{\infty}^i\,\vp_t\,dxdt \qquad \forall 1\leq i\leq k.
\end{equation}
By an argument similar to the proofs of Lemma 18.4 and Lemma 18.6 of \cite{Va1}, we can get
\begin{equation}\label{eq-intial-condition-of-u-i-infty-to-volume-times-dirac-delta}
u^i_{0,\lambda}(x)\to M_i\,\delta(x) \quad \mbox{as $\lambda\to \infty$} \qquad \mbox{and} \qquad u^i_{\infty}(x,t)\to M_i\,\delta_0(x) \quad \mbox{as $t\to 0$} \qquad \forall 1\leq i\leq k.
\end{equation}
By \eqref{eq-weak-formula-of-u-i-infty-10} and \eqref{eq-intial-condition-of-u-i-infty-to-volume-times-dirac-delta}, $u^i_{\infty}$, $(i=1,\cdots,k)$, is a weak solution of \eqref{eq-main-equation-of-system}  with initial data $M_i\delta$. Then by \eqref{eq-mass-conservation-of-rescaled-solution-u-i-lambda} and Theorem \ref{thm-convergence-of-v-defined-above-to-some-mass-times-barrenblatt-function},
\begin{equation}\label{eq-equiv-u-i-infty-to-L-1-mass-and-barrenblatt-solutio34}
u_{\infty}^i(x,t)=\frac{M_i}{\left|\bold{M}\right|}\mathcal{B}_{\left|\bold{M}\right|}(x,t) \qquad \forall (x,t)\in \R^n\times\left[0,\infty\right).
\end{equation}
By \eqref{eq-equiv-u-i-infty-to-L-1-mass-and-barrenblatt-solutio34} and an argument similar to the proof of Theorem 1.5 of \cite{KL3} we have \eqref{eq-L-1-convergence-between-u-i-and-M-i-barrenblatt-profiles} and \eqref{eq-L-infty-convergence-between-u-i-and-M-i-barrenblatt-profiles} and the theorem follows.
\end{proof}

\subsection{Asymptotic behaviour II : The Entropy Approach.}
In this subsection, we are going to consider another approach called entropy method to show the asymptotic large behaviour of function $u^i$, $\left(1,\cdots,k\right)$. Consider the continuous rescaling
\begin{equation}\label{eq-continuous-rescaling-of=solution-u-i-and-bold-u}
\theta^{\,i}\left(\eta,\tau\right)=t^{\,a_1}u^i\left(x,t\right) \qquad \mbox{and} \qquad \mbox{\boldmath$\theta$}\left(\eta,\tau\right)=t^{\,a_1}\bold{u}\left(x,t\right) \qquad\left(\,\eta=xt^{-a_2},\,\,\,\tau=\log t\,\right)
\end{equation}
where $a_1$ and $a_2$ are given by \eqref{eq-constant-alpha-1-beta-1-k}. Then $\mbox{\boldmath$\theta$}=\left(\theta^{\,1},\cdots,\theta^{\,k}\right)$ satisfies
\begin{equation}\label{eq-equation-for-soution-boldmath-theta}
\left(\theta^{\,i}\right)_{\tau}=\nabla\cdot\left(m\Theta^{m-1}\nabla\theta^{\,i}\right)+a_2\eta\cdot\nabla\theta^{\,i}+a_1\theta^{\,i}=\nabla\cdot\left(m\Theta^{m-1}\nabla\theta^{\,i}\right)+a_2\nabla\left(\eta\,\theta^{\,i}\right) \qquad \forall 1\leq i\leq k.
\end{equation}
where $\Theta=\left|\mbox{\boldmath$\theta$}\right|$. By \eqref{eq-equation-for-bold-v-sub-equal-standard-PME}, $\Theta$ also satisfies
\begin{equation}\label{eq-equation-for-Theta-continuous-rescaled-function-of-absolute-bold-u}
\Theta_{\tau}=\La\Theta^{m}+a_2\nabla\cdot\left(\eta\,\Theta\right)+m\Theta^{m-2}\left(\left|\nabla\Theta\right|^2-\sum_{i=1}^{k}\left|\nabla\theta^i\right|^2\right).
\end{equation}
We define the functional $H_{\Theta}$ by
\begin{equation}\label{eq-definition-of-functional-for-Theta-absolute-of-bold-theta}
H_{\Theta}\left(\tau\right)=\int_{\R^n}\left(\frac{1}{m-1}\Theta^m+\frac{a_2}{2}\left|\eta\right|^{\,2}\Theta\right)\,d\eta,
\end{equation}
Then we can get the following variation of the entropy.
\begin{lemma}
Let $\Theta$ be a bounded smooth solution of \eqref{eq-equation-for-Theta-continuous-rescaled-function-of-absolute-bold-u} and $H_{\Theta}$ be the functional given by \eqref{eq-definition-of-functional-for-Theta-absolute-of-bold-theta}  satisfying
\begin{equation*}
\int_{\R^n}\left(1+|x|^2\right)\left|\bold{u}_0\right|(x)\,dx<\infty.
\end{equation*}
Then
\begin{equation*}
\begin{aligned}
\frac{dH_{\Theta}}{d\tau}&=-I_1-I_2
\end{aligned}
\end{equation*}
where
\begin{equation*}
I_1=\int_{\R^n}\Theta\left|\nabla\left(\frac{m}{m-1}\Theta^{m-1}+\frac{a_2}{2}\left|\eta\right|^2\right)\right|^2\,d\eta\geq 0
\end{equation*}
and
\begin{equation*}
I_2=m\int_{\R^n}\Theta^{m-2}\left(\frac{m}{m-1}\Theta^{m-1}+\frac{a_2}{2}\left|\eta\right|^2\right)\left(\sum_{i=1}^{k}\left|\nabla\theta^i\right|^2-\left|\nabla\Theta\right|^2\right)\,d\eta\geq 0.
\end{equation*}
\end{lemma}
\begin{proof}
By integration by parts, we have
\begin{align}
\frac{dH_{\Theta}}{d\tau}&=\int_{\R^n}\left(\frac{m}{m-1}\Theta^{m-1}+\frac{a_2}{2}\left|\eta\right|^2\right)\Theta_{\tau}\,d\eta\notag\\
&=\int_{\R^n}\left(\frac{m}{m-1}\Theta^{m-1}+\frac{a_2}{2}\left|\eta\right|^2\right)\left(\La\Theta^{m}+a_2\nabla\cdot\left(\eta\,\Theta\right)\right)\,d\eta\notag\\
&\qquad \qquad +m\int_{\R^n}\Theta^{m-2}\left(\frac{m}{m-1}\Theta^{m-1}+\frac{a_2}{2}\left|\eta\right|^2\right)\left(\left|\nabla\Theta\right|^2-\sum_{i=1}^{k}\left|\nabla\theta^i\right|^2\right)\,d\eta\notag\\
&=-\int_{\R^n}\Theta\left|\nabla\left(\frac{m}{m-1}\Theta^{m-1}+\frac{a_2}{2}\left|\eta\right|^2\right)\right|^2\,d\eta\notag\\
&\qquad \qquad -\left(m\int_{\R^n}\Theta^{m-2}\left(\frac{m}{m-1}\Theta^{m-1}+\frac{a_2}{2}\left|\eta\right|^2\right)\left(\sum_{i=1}^{k}\left|\nabla\theta^i\right|^2-\left|\nabla\Theta\right|^2\right)\,d\eta\right)\label{eq-direct-computation-of-derivation-H-sub-Theta-w-r-t-tau}\\
&=-I_1-I_2.\notag
\end{align}
By Cauchy-Schwarz inequality,
\begin{equation}\label{eq-Cauchy-Schwarz-inequality-on-Theta}
\left|\nabla\Theta\right|^2=\frac{\left|\sum_{i=1}^{k}\nabla\left(\theta^i\right)^2\right|^2}{4\sum_{i=1}^{k}\left(\theta^i\right)^2}\leq\frac{\left[\sum_{i=1}^{k}\left(\theta^i\left|\nabla\theta^i\right|\right)\right]^2}{\sum_{i=1}^{k}\left(\theta^i\right)^2} \leq \sum_{i=1}^{k}\left|\nabla\theta^i\right|^2.
\end{equation}
Thus by \eqref{eq-direct-computation-of-derivation-H-sub-Theta-w-r-t-tau} and \eqref{eq-Cauchy-Schwarz-inequality-on-Theta}, the positivity of $I_2$ holds and the lemma follows.
\end{proof}

Let $\left\{\tau_n\right\}_{n=1}^{\infty}\subset \R^+$ be a sequence such that $\tau_k\to\infty$ as $k\to\infty$ and define
\begin{equation*}
\theta_n^{\,i}\left(\eta,\tau\right)=\theta^{\,i}\left(\eta,\tau+\tau_n\right)\qquad \forall 1\leq i\leq k\qquad \mbox{and}\qquad \Theta_n\left(\eta,\tau\right)=\Theta\left(\eta,\tau+\tau_n\right).
\end{equation*}
By an argument similar to the previous section, $\left\{\theta^{\,i}_n\right\}_{n=1}^{\infty}$ and $\left\{\Theta_n\right\}_{n=1}^{\infty}$ are precompact in $L^{\infty}_{loc}\left(0,\infty,L^1\left(\R^n\right)\right)$. Then letting $n\to\infty$, we have
\begin{equation}\label{eq-convergnece-of-theat0-i-and-Theata-with-tilde}
\theta^{\,i}_n(\eta,\tau)\to \tilde{\theta^{\,i}}(\eta,\tau)\qquad \mbox{and} \qquad \Theta_n(\eta,\tau)\to \tilde{\Theta}(\eta,\tau).
\end{equation}
Moreover, it can be easily checked that $\tilde{\mbox{\boldmath$\theta$}}=\left(\tilde{\theta}^{\,1},\cdots,\tilde{\theta}^{\,k}\right)$ is a weak solution of \eqref{eq-equation-for-soution-boldmath-theta} with the same estimates as that of $\mbox{\boldmath$\theta$}_n=\left(\theta_n^{\,1},\cdots,\theta_n^{\,k}\right)$.\\
\indent Putting $\lambda=t$, $x=\eta$ and $t=1$ in \eqref{eq-def-of-rescaled-function-u-lambda-i}, we can have the following families of functions with respect to $t>0$
\begin{equation}\label{eq-asymptotic-sequence-with-new-variables-on-Entropy}
\lambda^{a_1}u^i\left(\lambda^{a_2}x,\lambda t\right)=t^{\,a_1}u^i\left(t^{a_2}\eta,t\right)=\theta^{\,i}\left(\eta,\tau\right)\qquad \forall 1\leq i\leq k.
\end{equation}
Then by an argument similar to the proof of Theorem \ref{thm-asymptotic-behaviour-of-system-PME-L-1=and=L=infty}, there exists a function $\tilde{u}^i$ such that
\begin{equation*}\label{eq-convergence-of-families-with-lamdb=t-x=eta-t=1-to-function-independent-to-tau}
t^{\,a_1}u^i\left(t^{a_2}\eta,t\right)\to\tilde{u}^i\left(\eta,1\right) \qquad \mbox{as $t\to\infty$} \qquad \forall 1\leq i\leq k.
\end{equation*}
Combining this with \eqref{eq-convergnece-of-theat0-i-and-Theata-with-tilde}, we can get
\begin{equation*}
\theta_n^{\,i}\left(\eta,\tau\right)\to  \tilde{\theta}^{\,i}\left(\eta,\tau\right):=\tilde{\theta}^{\,i}\left(\eta\right) \qquad \forall 1\leq i\leq k
\end{equation*}
and
\begin{equation}\label{eq-convergence-under-new-scaling-on-Entropy}
\Theta_n\left(\eta,\tau\right)\to\tilde{\Theta}\left(\eta,\tau\right):=\tilde{\Theta}\left(\eta\right) \qquad \mbox{as $n\to\infty$} \qquad \mbox{in $L^1\cap L^{\infty}_{loc}$}.
\end{equation}

Here is the second proof (the entropy approach) of the asymptotic large time behaviour of the function $u^i$, $\left(1\leq i\leq k\right)$.
\begin{proof}[\textbf{Proof of Theorem \ref{thm-asymptotic-with-entrophy-methods}}]
By \eqref{eq-asymptotic-sequence-with-new-variables-on-Entropy} and \eqref{eq-convergence-under-new-scaling-on-Entropy}, it suffices to show that
\begin{equation}\label{eq-equality-between-tilde-Theta-eta-and-Barenblatt-at-eta-and-1}
\tilde{\Theta}\left(\eta\right)=\widetilde{\mathcal{B}_{\left|\bold{M}\right|}}\left(\eta\right)\qquad \mbox{and}\qquad \tilde{\theta}^{\,i}\left(\eta\right)=\frac{M_i}{\left|\bold{M}\right|}\widetilde{\mathcal{B}_{\left|\bold{M}\right|}}\left(\eta\right) \qquad \forall 1\leq i\leq k
\end{equation}
where $\left\|u_0^i\right\|_{L^1\left(\R^n\right)}=M_i$ for $i=1,\cdots, k$ and $\bold{M}=\left(M_1,\cdots,M_k\right)$. \\
\indent Replacing $\Theta$ and $\theta^{\,i}$ by $\Theta_n$ and $\theta_n^{\,i}$, respectively, and letting $n\to\infty$ in \eqref{eq-direct-computation-of-derivation-H-sub-Theta-w-r-t-tau}, we have
\begin{equation}\label{eq-derivative-of-functional-of-tilde-Theta-w-r-t-tau-from-limit}
\begin{aligned}
\frac{dH_{\tilde{\Theta}}}{d\tau}&=-\int_{\R^n}\tilde{\Theta}\left|\nabla\left(\frac{m}{m-1}\tilde{\Theta}^{m-1}+\frac{a_2}{2}\left|\eta\right|^2\right)\right|^2\,d\eta\\
&\qquad \qquad +m\int_{\R^n}\tilde{\Theta}^{m-2}\left(\frac{m}{m-1}\tilde{\Theta}^{m-1}+\frac{a_2}{2}\left|\eta\right|^2\right)\left(\left|\nabla\tilde{\Theta}\right|^2-\sum_{i=1}^{k}\left|\nabla\tilde{\theta}^{\,i}\right|^2\right)\,d\eta.
\end{aligned}
\end{equation}
Note that the second term on the right hand side of \eqref{eq-derivative-of-functional-of-tilde-Theta-w-r-t-tau-from-limit} is negative by Cauchy-Schwarz inequality,
\begin{equation}\label{eq-Cauchy-Schwarz-inequal-for-tilde-theta-and-Theta}
\left|\nabla\tilde{\Theta}\right|^2=\frac{\left|\sum_{i=1}^{k}\nabla\left(\tilde{\theta}^{\,i}\right)^2\right|^2}{4\sum_{i=1}^{k}\left(\tilde{\theta^{\,i}}\right)^2}\leq\frac{\left[\sum_{i=1}^{k}\left(\tilde{\theta^{\,i}}\left|\nabla\tilde{\theta^{\,i}}\right|\right)\right]^2}{\sum_{i=1}^{k}\left(\tilde{\theta^{\,i}}\right)^2} \leq \sum_{i=1}^{k}\left|\nabla\tilde{\theta^{\,i}}\right|^2.
\end{equation}
On the other hand, by \eqref{eq-convergence-under-new-scaling-on-Entropy}  we have
\begin{equation}\label{eq-zero-of-derivatvies-functional-under-tilde-Theta-w-r-t-tau}
\frac{dH_{\tilde{\Theta}}}{d\tau}=0.
\end{equation}
Then by \eqref{eq-derivative-of-functional-of-tilde-Theta-w-r-t-tau-from-limit} and \eqref{eq-zero-of-derivatvies-functional-under-tilde-Theta-w-r-t-tau}, we have
\begin{equation}\label{eq-two-equations-from-equibrilium-of-orbit-of-rescaled-functions-theta-and-Theta-1}
\frac{m}{m-1}\tilde{\Theta}^{m-1}+\frac{a_2}{2}\left|\eta\right|^2=\frac{m}{m-1}C_1
\end{equation}
for some constant $C_1>0$ and
\begin{equation}\label{eq-two-equations-from-equibrilium-of-orbit-of-rescaled-functions-theta-and-Theta-2}
\left|\nabla\tilde{\Theta}\right|^2=\sum_{i=1}^{k}\left|\nabla\tilde{\theta}^{\,i}\right|^2 \qquad \mbox{a.e. on $\R^n$}.
\end{equation}
By the positivity of $\tilde{\Theta}$ and \eqref{eq-two-equations-from-equibrilium-of-orbit-of-rescaled-functions-theta-and-Theta-1}, $\tilde{\Theta}$ can be represented by
\begin{equation*}
\tilde{\Theta}\left(\eta\right)=\left(C_1-\frac{a_2(m-1)}{2m}\left|\eta\right|^2\right)^{\frac{1}{m-1}}_+=\left(C_1-\frac{a_1(m-1)}{2mn}\left|\eta\right|^2\right)^{\frac{1}{m-1}}_+.
\end{equation*}
Since $\tilde{\Theta}$ has $L^1$ mass $\left|\bold{M}\right|$, $C_1=\mathcal{C}_{\left|\bold{M}\right|} $ and the first equality of \eqref{eq-equality-between-tilde-Theta-eta-and-Barenblatt-at-eta-and-1} follows. \\ 
\indent By \eqref{eq-Cauchy-Schwarz-inequal-for-tilde-theta-and-Theta} and \eqref{eq-two-equations-from-equibrilium-of-orbit-of-rescaled-functions-theta-and-Theta-2}, we have
\begin{equation}\label{eq-conditions-from-relation-between-nabla-tilde-Theta-and-sum-of-nabla-tilde-theta-1}
\nabla\tilde{\theta}^{\,i}\cdot\nabla\tilde{\theta}^{\,j}=\left|\nabla\tilde{\theta}^{\,i}\right|\left|\nabla\tilde{\theta}^{\,j}\right|\qquad \Rightarrow\qquad \nabla\tilde{\theta}^{\,i}\parallel\nabla\tilde{\theta}^{\,j} \qquad \forall 1\leq i<j\leq k
\end{equation}
and, by the condition of equality in the Cauchy-Schwarz inequality,
\begin{equation}\label{eq-conditions-from-relation-between-nabla-tilde-Theta-and-sum-of-nabla-tilde-theta-2}
\tilde{\theta}^{\,i}=C_2\left|\nabla\tilde{\theta}^{\,i}\right|\qquad \forall 1\leq i\leq k.
\end{equation}
for some constant $C_2>0$. By \eqref{eq-conditions-from-relation-between-nabla-tilde-Theta-and-sum-of-nabla-tilde-theta-1} and \eqref{eq-conditions-from-relation-between-nabla-tilde-Theta-and-sum-of-nabla-tilde-theta-2}, there exists a constant $C_i>1$ depending on $1\leq i\leq k$ such that
\begin{equation}\label{eq-representation-of-tilde-theta-sub-eta-square}
\tilde{\Theta}^2=\sum_{j=1}^{k}\left(\tilde{\theta}^{\,j}\right)^2=C^2_2\sum_{j=1}^{k}\left|\nabla\tilde{\theta}^{\,j}\right|^2=C_2^2C^2_i\left|\nabla\tilde{\theta}^{\,i}\right|^2=C_i^2\left(\tilde{\theta}^{\,i}\right)^2\qquad \forall \eta\in\R^n\bs\{0\}.
\end{equation}
By \eqref{eq-equality-between-tilde-Theta-eta-and-Barenblatt-at-eta-and-1} and \eqref{eq-representation-of-tilde-theta-sub-eta-square},
\begin{equation*}
\tilde{\theta}^{\,i}(\eta)=\frac{1}{C_i}\tilde{\Theta}(\eta)=\frac{1}{C_i}\mathcal{B}_{\left|\bold{M}\right|}\left(\eta,1\right).
\end{equation*}
By the conservation of $L^1$ mass of $u^{i}(t)$, it can be easily checked that the constant $C_i=\frac{\left|\bold{M}\right|}{M_i}$.  This immediately implies the second equality of \eqref{eq-equality-between-tilde-Theta-eta-and-Barenblatt-at-eta-and-1} and the theorem follows.
\end{proof}

\section{Harnack Type Inequality of Degenerated System}
\setcounter{equation}{0}
\setcounter{thm}{0}


In this section, we shall find a suitable Harnack type inequality for the function $u^i$ of \eqref{eq-main-equation-of-system} which makes the spatial average of $u^i$ under control by the value of $u^i$ at one point. We first review a Harnack-type estimate of the porous medium equation
\begin{equation*}\label{eq-porous-medium-equation-in-section-waiting-time}
u_t=\La u^m. 
\end{equation*}
\begin{lemma}[cf. Theorem 3.1 of \cite{AC}]\label{lem-harnack-type-estimate-by-A-and-Caffarelli}
Let $u$ be a continuous solution of the porous medium equation in $\R^n\times\left[0,T\right]$ for some $T>0$. Then there exists a constant $C=C\left(m,n\right)>0$ such that
\begin{equation*}
\int_{|x|<R}u(x,0)\,dx\leq C\left(\frac{R^{n+\frac{2}{m-1}}}{T^{\frac{1}{m-1}}}+T^{\frac{n}{2}}u^{\frac{(m-1)n+2}{2}}\left(0,T\right) \right).
\end{equation*}
\end{lemma}

Let $\mathcal{B}_M(x,t)$ be the self-similar Barenblatt solution of the porous medium equation with $L^1$ mass $M$. Then for any $\tau>0$ the function $\mathcal{B}_M\left(x,t+\tau\right)$ is also a solution of porous medium equation in $\R^n\times\left[0,\infty\right)$. Thus by Lemma \ref{lem-harnack-type-estimate-by-A-and-Caffarelli}, we can get
\begin{equation}\label{eq-harnack-type-estimates-for-barrenblatt-with-tau-translations}
\int_{|x|<R}\mathcal{B}_{M}(x,\tau)\,dx\leq C\left(\frac{R^{n+\frac{2}{m-1}}}{T^{\frac{1}{m-1}}}+T^{\frac{n}{2}}\mathcal{B}_M^{\frac{(m-1)n+2}{2}}\left(0,T+\tau\right) \right).
\end{equation}
Letting $\tau\to 0$ in \eqref{eq-harnack-type-estimates-for-barrenblatt-with-tau-translations}, we can get the following Harnack type estimate for Barrenblatt solution of the porous medium equation
\begin{equation}\label{eq-harnack-type-estimates-for-barrenblatt}
\int_{|x|<R}\mathcal{B}_{M}(x,0)\,dx\leq C\left(\frac{R^{n+\frac{2}{m-1}}}{T^{\frac{1}{m-1}}}+T^{\frac{n}{2}}\mathcal{B}_M^{\frac{(m-1)n+2}{2}}\left(0,T\right) \right).
\end{equation}

For any $M>0$, denote by $P\left(M\right)$ the class of all non-negative continuous weak solution $\bold{w}=\left(w^1,\cdots,w^k\right)$ of 
\begin{equation*}
\left(w^i\right)_t=\nabla\cdot\left(\left|\bold{w}\right|^{m-1}\nabla w^i\right) \qquad \mbox{in $\R^n\times\left[0,\infty\right)$}
\end{equation*}
satisfying
\begin{equation*}
\sup_{t>0}\int_{\R^n}\left|\bold{w}\right|(x,t)\,dx\leq M.
\end{equation*}

Using the inequality \eqref{eq-harnack-type-estimates-for-barrenblatt}, we will give a  proof of our Harnack type inequality.

\begin{proof}[\textbf{Proof of Theorem \ref{lem-calculation-about-waiting-time-lower-bound}}]
Let $k\in\N$ and $T>0$ be fixed. By Lemma \ref{lem-conservation-law-of-L-1-mass}, there exists a constant $M_{\bold{u}}>0$, depending on $\bold{u}\left(\cdot,0\right)$, such that
\begin{equation}\label{eq-bound-of-absolute-of-bold-u-in-L-1-for-all-time-sup}
\int_{\R^n}\left|\bold{u}\right|(x,t)\,dx\leq M_{\bold{u}} \qquad \forall 0\leq t\leq T.
\end{equation}
By Lemma \ref{lem-uniquness-of-weak-solution-2} and \eqref{eq-bound-of-absolute-of-bold-u-in-L-1-for-all-time-sup}, we are going to show that \eqref{eq-claim-L-1-norm=of-u-at-initial-is-bounded-by-H-sub-m} holds when $\bold{u}\in P\left(M\right)$ for some constant $M>0$.\\
\indent We divide the proof into two cases.\\
\textbf{Case 1.} $\textbf{supp} \left|\bold{u}\right|(\cdot,0)\subset B_1=\left\{x\in\R^n:|x|\leq 1\right\}$.\\
\indent If \eqref{eq-claim-L-1-norm=of-u-at-initial-is-bounded-by-H-sub-m} is violated, then for each $j\in\N$ there exist a constant $R_j>0$, a solution $\bold{u}_j=\left(u^1_j,\cdots,u^{k}_j\right)\in P\left(M_j\right)$ and a number $1\leq i'(j)\leq k$ such that 
\begin{equation}\label{eq-asumption-if-eq-claim-L-1-norm=of-u-at-initial-is-bounded-by-H-sub-m-violated}
\int_{|x|<R_j}u^{i\,'}_j(x,0)\,dx\geq\frac{j}{\left(\mu^i_j\right)^{\frac{(m-1)n+2}{2}}}\left(\frac{R_j^{n+\frac{2}{m-1}}}{T^{\frac{1}{m-1}}}+T^{\frac{n}{2}}\left(u_j^{i'}\right)^{\frac{(m-1)n+2}{2}}\left(0,T\right)\right)
\end{equation}
where
\begin{equation*}
\mu_j^{\,i\,'}=\frac{\int_{\R^n}u_j^{\,i\,'}(x,0)\,dx}{\max_{1\leq l\leq k}\left\{\int_{\R^n}u_j^l(x,0)\,dx\right\}} \qquad \forall 1\leq i\leq k.
\end{equation*}
Without loss of generality we may assume that $i\,'=1$. For each $j\in\N$, let
\begin{equation}\label{eq-definition-of-I-i-L-1-mass-of-u-i-max}
I_j=\max_{1\leq l\leq k}\left\{\int_{\R^n}u^{l}_j(x,0)\,dx\right\}
\end{equation}
and consider the rescaled function 
\begin{equation*}
v^i_j(x,t)=\frac{1}{I_j^{\frac{2}{(m-1)n+2}}}u^i_j\left(I_j^{\frac{m-1}{(m-1)n+2}}x,t\right) \qquad \forall 1\leq i\leq k.
\end{equation*}
By direct computation, one can easily check that $\bold{v}_j=\left(v_j^1,\cdots,v^{k}_j\right)$ is also a solution of \eqref{eq-system-PME-for-waiting-time-1} in $\R^n\times[0,\infty)$ with 
\begin{equation}\label{eq-L-1-mass-of-rescaled-function-v-i-s}
0<\mu^0\leq \mu_j^i=\int_{\R^n}v^{i}_{j}(x,0)\,dx\leq 1 \qquad \forall 1\leq i\leq k,\,j\in\N.
\end{equation}  
By \eqref{eq-asumption-if-eq-claim-L-1-norm=of-u-at-initial-is-bounded-by-H-sub-m-violated}, \eqref{eq-definition-of-I-i-L-1-mass-of-u-i-max} and \eqref{eq-L-1-mass-of-rescaled-function-v-i-s}, 
\begin{equation*}
I_j\to\infty \qquad \mbox{ as $j\to\infty$}
\end{equation*} 
since $R_j^{n+\frac{2}{m-1}}T^{-\frac{1}{m-1}}\geq T^{\frac{n}{2}}>0$. Thus, for each $j\in\N$
\begin{equation*}
\textbf{supp}\,v_j^{\,i}(\cdot, 0)\subset B_{\frac{1}{I_j^{\,\,\frac{m-1}{(m-1)n+2}}}}\left(0\right) \qquad \forall 1\leq i\leq k.
\end{equation*}
By the Ascoli theorem and a diagonalization argument the sequence $\left\{\bold{v}_j\right\}_{j=1}^{\infty}$ has a subsequence which we may assume without loss of generality to be the sequence itself such that
\begin{equation}\label{eq-convergence-opf-int-v-j-i-to-mass-mu-i-1}
v_j^{\,i}(x,0)\to \mu^{\,i}\delta \qquad \forall 1\leq i\leq k
\end{equation}
where $\delta$ is Dirac's delta function and $\mu^{\,i}$ is a constant such that $0<\mu^0\leq \mu^{\,i}\leq 1$. Then, by the Theorem \ref{thm-asymptotic-behaviour-of-system-PME-L-1=and=L=infty},
\begin{equation}\label{eq-convergence-of-scaled-v-j-to-Barrenblatt-with-mass-absoluiton-bold-M}
v_j^i\to \frac{\mu^{\,i}}{\mbox{\boldmath$\mu$}}\mathcal{B}_{\mbox{\boldmath$\mu$}}\qquad \mbox{uniformly on compact subset of $\R^n\times(0,\infty)$}
\end{equation}
where $\mbox{\boldmath$\mu$}=\sqrt{\sum_{i=1}^k\left(\mu^{\,i}\right)^2}$ and $\mathcal{B}_{\mbox{\boldmath$\mu$}}$ is the self-similar Barenblatt solution of porous medium equation with $L^1$ mass $\mbox{\boldmath$\mu$}$. By \eqref{eq-L-1-mass-of-rescaled-function-v-i-s}, \eqref{eq-convergence-opf-int-v-j-i-to-mass-mu-i-1} and \eqref{eq-convergence-of-scaled-v-j-to-Barrenblatt-with-mass-absoluiton-bold-M}, there exists a number $j_0\in\N$ such that
\begin{equation}\label{eq-comparison-between-mu-i-j-and-mu-i-1}
\mu^1_j\geq \frac{1}{2}\mu^1\qquad \forall j\geq j_0
\end{equation}
and
\begin{equation}\label{eq-compare-between-v-j-1-and-barrenblat-with-M-1-by-converging-1}
\int_{|x|<I_j^{-\frac{m-1}{(m-1)n+2}}R_j}v^1_j(x,0)\,dx\leq 2\mu^1\leq 2\int_{|x|<I_j^{-\frac{m-1}{(m-1)n+2}}R_j}\mathcal{B}_{\mbox{\boldmath$\mu$}}(x,0)\,dx  \qquad \forall j\geq j_0
\end{equation}
and
\begin{equation}\label{eq-compare-between-v-j-1-and-barrenblat-with-M-1-by-converging-2}
\frac{\mu^1}{\mbox{\boldmath$\mu$}}\mathcal{B}_{\mbox{\boldmath$\mu$}}(0,T)\leq 2v_j^1\left(0,T\right)=\frac{2}{I_j^{\frac{2}{(m-1)n+2}}}u^1_j(0,T) \qquad \forall j\geq j_0.
\end{equation}
By \eqref{eq-harnack-type-estimates-for-barrenblatt}, \eqref{eq-comparison-between-mu-i-j-and-mu-i-1}, \eqref{eq-compare-between-v-j-1-and-barrenblat-with-M-1-by-converging-1} and \eqref{eq-compare-between-v-j-1-and-barrenblat-with-M-1-by-converging-2}
\begin{equation*}
\begin{aligned}
\frac{1}{I_j}\int_{|x|<R_j}u^1_j(x,0)\,dx&=\int_{|x|<I_j^{-\frac{m-1}{(m-1)n+2}}R_j}v^1_j(x,0)\,dx\\
&=2\int_{|x|<I_j^{-\frac{m-1}{(m-1)n+2}}R_j}\mathcal{B}_{\mbox{\boldmath$\mu$}}(x,0)\,dx\leq 2C\left(\frac{R_j^{n+\frac{2}{m-1}}}{I_j\,T^{\frac{1}{m-1}}}+T^{\frac{n}{2}}\mathcal{B}^{\frac{(m-1)n+2}{2}}_{\mbox{\boldmath$\mu$}}(0,T)\right)\\
&\leq\frac{2^{\frac{(m-1)n+4}{2}}k^{\frac{(m-1)n+2}{4}}C}{\left(\mu^1\right)^{\frac{(m-1)n+2}{2}}I_j}\left(\frac{R_j^{n+\frac{2}{m-1}}}{T^{\frac{1}{m-1}}}+T^{\frac{n}{2}}\left(u^1_j\right)^{\frac{(m-1)n+2}{2}}(0,T)\right)\\
&\leq\frac{2^{(m-1)n+4}k^{\frac{(m-1)n+2}{4}}C}{\left(\mu_j^1\right)^{\frac{(m-1)n+2}{2}}I_j}\left(\frac{R_j^{n+\frac{2}{m-1}}}{T^{\frac{1}{m-1}}}+T^{\frac{n}{2}}\left(u^1_j\right)^{\frac{(m-1)n+2}{2}}(0,T)\right) \qquad \forall j\geq j_0.
\end{aligned}
\end{equation*}
Hence, for $C_1=2^{(m-1)n+4}k^{\frac{(m-1)n+2}{4}}C$
\begin{equation*}
\begin{aligned}
\int_{|x|<R_j}u^1_j(x,0)\,dx&\leq \frac{C_1}{\left(\mu^1_j\right)^{\frac{(m-1)n+2}{2}}}\left(\frac{R_j^{n+\frac{2}{m-1}}}{T^{\frac{1}{m-1}}}+T^{\frac{n}{2}}\left(u^1_j\right)^{\frac{(m-1)n+2}{2}}(0,T)\right)  \qquad \forall j\geq j_0
\end{aligned}
\end{equation*}
, which contradicts \eqref{eq-claim-L-1-norm=of-u-at-initial-is-bounded-by-H-sub-m} and  the case follows.\\
\textbf{Case 2.} {\it General case} $\left|\bold{u}\right|\left(\cdot,0\right)$ has compact support.\\ 
Let $R_0>1$ be a constant such that
\begin{equation*}
\left|\bold{u}\right|(x,0)=0 \qquad \forall x\in\R^n\bs B_{R_0}
\end{equation*}
and consider the rescaled function
\begin{equation*}
w^i(x,t)=\frac{1}{R_0^{\frac{2}{m-1}}}u^i\left(R_0x,t\right) \qquad \forall 1\leq i\leq k.
\end{equation*}
Then $\bold{w}=\left(w^1,\cdots,w^k\right)$ is a solution of \eqref{eq-system-PME-for-waiting-time-1} with 
\begin{equation*}
\textbf{supp}\left|\bold{w}\right|(\cdot,0)\subset B_1. 
\end{equation*}
Then by the \textbf{Case 1},
\begin{equation}\label{eq-harnack-type-estimates-applying-case-1-for-general-case}
\begin{aligned}
\frac{1}{R_0^{n+\frac{2}{m-1}}}\int_{|x|<R}u^i(x,0)\,dx&=\int_{|x|<\frac{R}{R_0}}w^i(x,0)\,dx \\
&\leq \frac{C}{\left(\mu^i\right)^{\frac{(m-1)n+2}{2}}}\left(\frac{1}{T^{\frac{1}{m-1}}}\frac{R^{n+\frac{2}{m-1}}}{R_0^{n+\frac{2}{m-1}}}+T^{\frac{n}{2}}\left|\bold{w}\right|^{\frac{(m-1)n+2}{2}}\left(0,T\right)\right)\\
&= \frac{C}{\left(\mu^i\right)^{\frac{(m-1)n+2}{2}}}\left(\frac{1}{T^{\frac{1}{m-1}}}\frac{R^{n+\frac{2}{m-1}}}{R_0^{n+\frac{2}{m-1}}}+\frac{T^{\frac{n}{2}}}{R_0^{n+\frac{2}{m-1}}}\left|\bold{u}\right|^{\frac{(m-1)n+2}{2}}\left(0,T\right)\right) \qquad \forall 1\leq i\leq k
\end{aligned}
\end{equation}
for any $0<R<R_0$. Multiplying \eqref{eq-harnack-type-estimates-applying-case-1-for-general-case} by $R_0^{n+\frac{2}{m-1}}$, \eqref{eq-claim-L-1-norm=of-u-at-initial-is-bounded-by-H-sub-m} holds and the theorem follows.
\end{proof}

\section{1-Directional Travelling Wave Solution}
\setcounter{equation}{0}
\setcounter{thm}{0}

This section is devoted to find an 1-directional travelling wave type solutions and properties of solutions which has travelling wave behaviour at infinity. We start by considering the simplest case: 1-dimensional travelling wave.\\
\indent Let  $0\leq l\leq k$ be a integer and let $\bold{u}=\left(u^1,\cdots,u^k\right)$ be a 1-dimension travelling wave type solution of 
\begin{equation}\label{eq-main-system-eq-for-travelling-wave-type}
\left(u^i\right)_t=\nabla\cdot\left(m\left|\bold{u}\right|^{m-1}\nabla u^i\right) \qquad \forall (x,t)\in\R\times\R,
\end{equation}
i.e.,
\begin{equation}\label{eq-expression-of-soluiton-u-to-i-by-g-single-variable-function-travelling-wave-eq}
u^i(x,t)=g^i\left(x+c^it\right) \qquad \forall x\in\R,\,\,1\leq i\leq k
\end{equation}
for some constants $c^1$, $\cdots$, $c^k$ where
\begin{equation}\label{eq-assumption-for-constant-in-1-dimensional-solution}
c^1,\cdots, c^l<0\qquad \mbox{and} \qquad c^{l+1},\cdots, c^k>0
\end{equation}
and single variable functions $g^1$, $\cdots$, $g^k$ where
\begin{equation}\label{eq-condition-of-g-i-1-first-l-th-functions}
g^{i_1}(s)=0 \qquad  \mbox{for $s\geq 0$}, \qquad 
g^{i_1}(s)>0\qquad  \mbox{for $s<0$} \qquad \forall 1\leq i_1\leq l 
\end{equation}
and
\begin{equation}\label{eq-condition-of-g-i-1-last-k-l-th-functions}
g^{i_2}(s)>0\qquad \mbox{for $s>0$}\qquad g^{i_2}(s)=0 \qquad \mbox{for $s\leq 0$} \qquad \forall l+1\leq i_2\leq k .
\end{equation}
By \eqref{eq-main-system-eq-for-travelling-wave-type} and \eqref{eq-expression-of-soluiton-u-to-i-by-g-single-variable-function-travelling-wave-eq}, 
\begin{equation}\label{eq-equation-for-g-i-s-traveling-wave-equation-of-system}
c^ig^i\left(x+c^it\right)=\left(m\sum_{j=1}^{k}\left(g^{i}\right)^2(x+c^jt)\right)^{\frac{m-1}{2}}\left(g^i\right)'(x+c^it) \qquad \forall x\in\R,\,\,t\in\R,\,\,1\leq i\leq k.
\end{equation}

We first prove that there are only 1-directional travelling wave type solutions with components whose directions are the same.

\begin{lemma}\label{thm-condition-of-l-for-existence-of-travelling-wave-type-eq}
Let  $0\leq l\leq k$ be a positive integer and suppose that
\begin{equation*}
\begin{aligned}
\bold{u}(x,t)&=\left(u^1(x,t),\cdots,u^k(x,t)\right)=\left(g^1\left(x+c^1t\right),\cdots,g^{k}\left(x+c^kt\right)\right) \qquad \forall x\in\R,\,\,t\in\R
\end{aligned}
\end{equation*}
is a 1-dimension travelling wave type solution of \eqref{eq-main-system-eq-for-travelling-wave-type} which satisfies \eqref{eq-assumption-for-constant-in-1-dimensional-solution}, \eqref{eq-condition-of-g-i-1-first-l-th-functions} and \eqref{eq-condition-of-g-i-1-last-k-l-th-functions}. Then
\begin{equation}\label{eq-condition-of-l-in-travelling-wave-type-solution}
l=0 \qquad \mbox{or} \qquad l=k.
\end{equation}
Moreover, the constants $c^1$, $\cdots$, $c^k$ satisfy
\begin{equation}\label{eq-equality-of-constants-in-travelling-solutions}
c^1=\cdots=c^k.
\end{equation}
\end{lemma}

\begin{proof}
Suppose that $0< l<k$. By \eqref{eq-condition-of-g-i-1-first-l-th-functions}, \eqref{eq-condition-of-g-i-1-last-k-l-th-functions} and \eqref{eq-equation-for-g-i-s-traveling-wave-equation-of-system}, $\bold{g}_1=\left(g^{1},\cdots,g^{l}\right)$ and $\bold{g}_2=\left(g^{l+1},\cdots,g^{k}\right)$ are solutions of \eqref{eq-main-system-eq-for-travelling-wave-type} on the region where $x<0$, $t\leq 0$ and the region where $x>0$, $t\leq 0$, respectively. Since
\begin{equation*}
\bigcap_{i_1=1}^{l}\textbf{supp}\,g^{i_1}\left(\cdot+c^{i_1}t\right)\not=\emptyset\qquad \mbox{and} \qquad  \bigcap_{i_2=l+1}^{k}\textbf{supp}\,g^{i_2}\left(\cdot+c^{i_2}t\right)\not=\emptyset
\end{equation*}
for each $t\leq 0$, by Lemma \ref{eq-lem-equailty-of-support-ofu-i-and-u-j-when-initially-joined} we have
\begin{equation*}
\textbf{supp}\,g^{1}\left(\cdot+c^{1}t\right)=\cdots=\textbf{supp}\,g^{l}\left(\cdot+c^{l}t\right) \qquad \forall t\leq 0
\end{equation*}
and
\begin{equation*}
\textbf{supp}\,g^{l+1}\left(\cdot+c^{l+1}t\right)=\cdots=\textbf{supp}\,g^{k}\left(\cdot+c^{k}t\right) \qquad \forall t\leq 0.
\end{equation*}
Thus
\begin{equation}\label{eq-equality-of-constants-c-1-to-c-l-and-c-l=1-c-k-respectively}
c^1=\cdots=c^l<0 \qquad \mbox{and} \qquad c^{l+1}=\cdots=c^k>0.
\end{equation}
Let
\begin{equation*}\label{eq-constants-widehat-c-and-widetilde-c-speed-of-opposit-directions}
-\widehat{c}=c^1=\cdots=c^l<0 \qquad \mbox{and} \qquad \widetilde{c}=c^l=\cdots=c^k>0.
\end{equation*}
On the region where $\left\{t<0\right\}$, $\bold{g}_1=\left(g^{1},\cdots,g^{l}\right)$ satisfies
\begin{align}
&\left(m\left|\bold{g}_1\right|^{m-1}\left(g^{i_1}\right)'\right)'=-\widehat{c}\left(g^{i_1}\right)'\qquad \qquad \forall 1\leq i_1\leq l\notag\\
&\Rightarrow \qquad m\left|\bold{g}_1\right|^{m-1}\left(g^{i_1}\right)'=-\widehat{c}g^{i_1} \qquad \qquad \forall 1\leq i_1\leq l.\label{eq-ODE-for-bold-g-1-and-bold-g-2}\\
&\Rightarrow \qquad \frac{\left(g^{i_1}\right)'}{g^{i_1}}=\frac{\left(g^{j_1}\right)'}{g^{j_1}}\qquad \qquad \forall 1\leq i_1,\,j_1\leq l\label{eq-relation-between-g-i-1-and-g-j-1-in-1-to-l}.
\end{align}
Thus, for any $1\leq i_1\leq j_1\leq l$ there exists a constant $a^{i_1j_i}>0$ such that
\begin{equation}\label{eq-relation-between-f-i-s}
g^{j_1}=a^{i_1j_1}g^{i_1}.
\end{equation}
for some constants $a^{i_1j_1}>0$. By \eqref{eq-ODE-for-bold-g-1-and-bold-g-2} and \eqref{eq-relation-between-f-i-s},
\begin{equation}\label{eq-explicit-form-of-g-i-1-travelling-wave-eq}
A_{i_1}\left(\left(g^{i_1}\right)^{m-1}\right)'+1=0 \qquad \Rightarrow \qquad g^{i_1}(s)=\left(-\frac{s}{A_{i_1}}\right)_+^{\frac{1}{m-1}} \qquad \forall s\in\R,\,\,1\leq i\leq l.
\end{equation}
where $A_{i_1}=\frac{m\left(\sum_{j_1=1}^l\left(a^{i_1j_1}\right)^2\right)^{\frac{m-1}{2}}}{\widehat{c}(m-1)}$. Similarly
\begin{equation}\label{eq-explicit-form-of-g-i-2-travelling-wave-eq}
g^{i_2}(s)=\left(\frac{s}{A_{i_2}}\right)_+^{\frac{1}{m-1}} \qquad \forall s\in\R,\,\,l+1\leq i_2\leq k
\end{equation}
where $A_{i_2}=\frac{m\left(\sum_{j_2=l+1}^{k}\left(a^{i_2j_2}\right)^2\right)^{\frac{m-1}{2}}}{\widetilde{c}(m-1)}$. On the other hand,  
\begin{equation}\label{eq-equivalence-on-the-region-mixed-traveling-wave-2}
\begin{aligned}
-\widetilde{c}\,\frac{\left(g^{i_1}\right)'}{\,g^{i_1}}\left(x-\widehat{c}t\right)=\widehat{c}\,\frac{\left(g^{i_2}\right)'}{\,g^{i_2}}\left(x+\widetilde{c}t\right)\qquad \qquad \forall (x,t)\in\left\{x+\widetilde{c}t>0\right\}\cap\left\{x-\widehat{c}t<0\right\}.
\end{aligned}
\end{equation}
By \eqref{eq-equivalence-on-the-region-mixed-traveling-wave-2} and standard O.D.E. theory, we can get
\begin{equation*}
\begin{aligned}
\frac{\partial}{\partial x}\left(\log \left(g^{i_1}(x-\widehat{c}t)\right)^{\widetilde{c}}\left(g^{i_2}(x+\widetilde{c}t)\right)^{\widehat{c}}\right)=0 \qquad \forall (x,t)\in\left\{x+\widetilde{c}t>0\right\}\cap\left\{x-\widehat{c}t<0\right\}.
\end{aligned}
\end{equation*}
Thus there exists a constant $a^{i_1i_2}>0$ depending only on $t\in\R$ such that
\begin{equation}\label{eq-multi-of-g-i-1-and-g-i-2-equal-function-of-t-from-ODE}
\left(g^{i_1}(x-\widehat{c}t)\right)^{\widetilde{c}}\left(g^{i_2}(x+\widetilde{c}t)\right)^{\widehat{c}}=a^{i_1i_2}\left(t\right) \qquad \forall x\in\R,\,\,t\geq 0.
\end{equation}
On the other hand, by \eqref{eq-explicit-form-of-g-i-1-travelling-wave-eq} and \eqref{eq-explicit-form-of-g-i-2-travelling-wave-eq} there exists a positive constant $A>0$ such that
\begin{equation}\label{eq-multi-of-g-i-1-and-g-i-2-equal-function-of-x-t-from-explicit-form}
\left(g^{i_1}(x-\widehat{c}t)\right)^{\widetilde{c}}\left(g^{i_2}(x+\widetilde{c}t)\right)^{\widehat{c}}=A\left[\left(\widehat{c}t-x\right)^{\widetilde{c}}\left(\widetilde{c}t+x\right)^{\widehat{c}}\right]^{\frac{1}{m-1}} \qquad \forall (x,t)\in\left\{x+\widetilde{c}t>0\right\}\cap\left\{x-\widehat{c}t<0\right\}
\end{equation}
which is depending on $x$ and $t$ since $\widehat{c}$ and $\widetilde{c}$ are positive. By \eqref{eq-multi-of-g-i-1-and-g-i-2-equal-function-of-t-from-ODE} and \eqref{eq-multi-of-g-i-1-and-g-i-2-equal-function-of-x-t-from-explicit-form} contradiction arises. Therefore \eqref{eq-condition-of-l-in-travelling-wave-type-solution} holds. By \eqref{eq-equality-of-constants-c-1-to-c-l-and-c-l=1-c-k-respectively}, \eqref{eq-equality-of-constants-in-travelling-solutions} also holds and the theorem follows.
\end{proof}

An a consequence of Theorem \ref{thm-condition-of-l-for-existence-of-travelling-wave-type-eq},  we can get the 1-drectional travelling wave type solution of \eqref{eq-main-system-eq-for-travelling-wave-type}.
\begin{proof}[\textbf{Proof of Theorem \ref{cor-expression-of-travelling-wave-equation-explicit-form}}]
By the change of coordinates and an argument similar to the proof of Lemma \ref{thm-condition-of-l-for-existence-of-travelling-wave-type-eq}, \eqref{eq-form-of-1-directional-travelling-wave-type-sol-1} or \eqref{eq-form-of-1-directional-travelling-wave-type-sol-2} holds and the theorem follows.
\end{proof}

\indent As an application of Theorem \ref{cor-expression-of-travelling-wave-equation-explicit-form}, we can also get the properties of 1-directional travelling wave-like solutions whose behaviour at infty is like travelling wave. Let $0\leq l\leq k$ and let $\bold{u}=\left(u^1,\cdots,u^k\right)$ be a 1-dimension solution of \eqref{eq-main-system-eq-for-travelling-wave-type} satisfying
\begin{equation}\label{eq-assumption-2-for-single-variable-function-in-1-dimensional-solution}
\begin{cases}
\begin{aligned}
u^{i_1}(x,t)&\to 0 \qquad \qquad \qquad \mbox{uniformly}\quad \mbox{ as $s_{i_1}=x+c^{i_1}t\to\infty$}\qquad \forall 1\leq i_1\leq l \\
 u^{i_1}(x,t)&\to a^{i_1}\left(-s_{i_1}\right)^{\frac{1}{m-1}} \qquad \mbox{uniformly}\quad \mbox{ as $s_{i_1}=x+c^{i_1}t\to-\infty$} \qquad \forall 1\leq i_1\leq l 
\end{aligned}
\end{cases}
\end{equation}
and
\begin{equation}\label{eq-assumption-3-for-single-variable-function-in-1-dimensional-solution}
\begin{cases}
\begin{aligned}
u^{i_2}(x,t)&\to a^{i_2}\left(s_{i_2}\right)^{\frac{1}{m-1}}\qquad \mbox{uniformly}\quad \mbox{as $s_{i_2}=x+c^{i_2}t\to\infty$} \qquad \forall l+1\leq i_2\leq k\\
u^{i_2}(x,t)&\to 0 \qquad \qquad \qquad \mbox{uniformly}\quad \mbox{as $s_{i_2}=x+c^{i_2}t\to-\infty$} \qquad \forall l+1\leq i_2\leq k 
\end{aligned}
\end{cases}
\end{equation}
for some constants $c^1$, $\cdots$, $c^k$ satisfying \eqref{eq-assumption-for-constant-in-1-dimensional-solution} and some positive constants $a^1$, $\cdots$, $a^k$.\\
\indent By scaling and an argument similar to the proof of Lemma \ref{thm-condition-of-l-for-existence-of-travelling-wave-type-eq}, we also have that there are only 1-dimensional travelling wave type solutions with components whose directions are the same.

\begin{thm}\label{thm-equality-of-speed-at-infty-asymptotic-travelling-wave-equation}
Let $0\leq l\leq k$ be an integer and suppose that $\bold{u}=\left(u^1,\cdots,u^k\right)$ is an 1-dimensional travelling wave-like solution of \eqref{eq-main-system-eq-for-travelling-wave-type} which has conditions \eqref{eq-assumption-2-for-single-variable-function-in-1-dimensional-solution} and \eqref{eq-assumption-3-for-single-variable-function-in-1-dimensional-solution} for some constants $c^1$, $\cdots$, $c^k$ satisfying \eqref{eq-assumption-for-constant-in-1-dimensional-solution} and some positive constants $a^1$, $\cdots$, $a^k>0$.   Then
\begin{equation}\label{eq-only-one-way-of-travelling-like-solutions}
l=0 \qquad \mbox{or} \qquad l=k.
\end{equation} 
Moreover, the constants $c^1$, $\cdots$, $c^k$ satisfy
\begin{equation}\label{eq-equality-of-constants-in-travelling-like-solutions}
c^1=\cdots=c^k.
\end{equation}
\end{thm}

\begin{proof}
For each $\epsilon>0$, let
\begin{equation*}
u_{\epsilon}^i(x,t)=\epsilon^{\frac{1}{m-1}}u^i\left(\frac{x}{\epsilon},\frac{t}{\epsilon}\right)\qquad \forall x\in\R,\,\,t\in\R,\,\, 1\leq i\leq k.
\end{equation*}
Then $\bold{u}_{\epsilon}=\left(u^1_{\epsilon},\cdots,u^k_{\epsilon}\right)$ is also a solution of  \eqref{eq-main-system-eq-for-travelling-wave-type} in $\R\times\R$. By \eqref{eq-assumption-2-for-single-variable-function-in-1-dimensional-solution} and \eqref{eq-assumption-3-for-single-variable-function-in-1-dimensional-solution}, the sequence $\left\{\bold{u}_{\epsilon}\right\}_{\epsilon>0}$ converges uniformly in $\R\times\R$ to an 1-dimensional travelling wave solution of \eqref{eq-main-system-eq-for-travelling-wave-type} as $\epsilon\to 0$. Thus by Theorem \ref{thm-condition-of-l-for-existence-of-travelling-wave-type-eq}, \eqref{eq-only-one-way-of-travelling-like-solutions}, \eqref{eq-equality-of-constants-in-travelling-like-solutions} holds and the theorem follows.
\end{proof}

We now can extend the result for the 1-directional travelling wave solution (Theorem \ref{cor-expression-of-travelling-wave-equation-explicit-form}) to the one for solutions whose behaviour at infinity is like the travelling wave.

\begin{cor}
Let $e\in\R^n$ and let $\bold{u}=\left(u^1,\cdots,u^k\right)$ be a 1-directional travelling wave-like wave solution of \eqref{eq-main-system-eq-for-travelling-wave-type} with respect to the direction $e$. Then for each $1\leq i\leq k$ there exist positive constants $\widehat{c}$, $\widetilde{c}$ and $a^i$ such that
\begin{equation*}
\begin{cases}
\begin{aligned}
u^i(x,t)&\to 0 \qquad \qquad \qquad \qquad \mbox{as $x\cdot e-\widehat{c}\,t\to\infty$}\\
u^i(x,t)&\to a^i\left(\widehat{c}t-x\cdot e\right)_+^{\frac{1}{m-1}} \qquad \mbox{as $x\cdot e-\widehat{c}\,t\to-\infty$}
\end{aligned}
\end{cases}
\end{equation*}
or
\begin{equation*}
\begin{cases}
\begin{aligned}
u^i(x,t)&\to a^i\left(\widetilde{c}t+x\cdot e\right)_+^{\frac{1}{m-1}} \qquad  \mbox{as $x\cdot e+\widetilde{c}\,t\to\infty$}\\
u^i(x,t)&\to 0 \qquad \qquad \qquad \qquad \mbox{as $x\cdot e+\widetilde{c}\,t\to-\infty$}.
\end{aligned}
\end{cases}
\end{equation*}
\end{cor}

\noindent {\bf Acknowledgement:} 
Ki-Ahm Lee has been supported by Samsung Science \& Technology Foundation (SSTF) under Project Number SSTF-BA1701-03. Ki-Ahm Lee also holds a joint appointment with the Research Institute of Mathematics of Seoul National University. Sunghoon Kim was supported by the Research Fund, 2020 of The Catholic University of Korea.

\end{document}